\documentclass[12pt]{amsart}
\usepackage{mathrsfs}
\usepackage{amssymb}
\usepackage{amsmath}
\usepackage{bbold}
\usepackage[all]{xy}
\renewcommand{\-}{\hbox{-}}

\newcommand{\Mod}{\operatorname{\bf Mod}\nolimits}
\newcommand{\bfmod}{\operatorname{\bf mod}\nolimits}

\newcommand{\seq}{\operatorname{\bf seq}\nolimits}
\newcommand{\tsseq}{\operatorname{{\mathcal{TS}}-\seq}\nolimits}
\newcommand{\ts}{\operatorname{{\mathcal{TS}}}\nolimits}

\newcommand{\Ad}{\operatorname{{\mathcal Ad}}\nolimits}
\newcommand{\hatR}{\operatorname{{\widehat R}}\nolimits}

\newcommand{\Hom}{\operatorname{Hom}\nolimits}
\newcommand{\bhom}{\operatorname{hom}\nolimits}
\newcommand{\PHom}{\operatorname{PHom}\nolimits}

\newcommand{\Id}{\operatorname{Id}\nolimits}
\newcommand{\Ker}{\operatorname{Ker}\nolimits}

\newcommand{\Ind}{\operatorname{Ind}\nolimits}
\newcommand{\Res}{\operatorname{Res}\nolimits}

\newcommand{\Tr}{\operatorname{Tr}\nolimits}

\def\CA{{\mathcal{A}}}

\def\CC{{\mathcal{C}}}
\def\CG{{\mathcal{G}}}
\def\CD{{\mathcal{D}}}
\def\CI{{\mathcal{I}}}
\def\CL{{\mathcal{L}}}
\def\thyph{{\text{-}}}
\newcommand{\CE}{\operatorname{\mathcal E}\nolimits}
\newcommand{\CF}{\operatorname{\mathcal F}\nolimits}
\newcommand{\CH}{\operatorname{\mathcal H}\nolimits}

\newcommand{\CK}{\operatorname{\mathcal K}\nolimits}

\newcommand{\bI}{\operatorname{\mathbb I}\nolimits}
\newcommand{\bR}{\operatorname{\mathbb R}\nolimits}
\newcommand{\boK}{\operatorname{\bf K}\nolimits}
\newcommand{\CM}{\operatorname{\mathcal M}\nolimits}
\newcommand{\CP}{\operatorname{\mathcal P}\nolimits}
\newcommand{\CR}{\operatorname{\mathfrak{R}}\nolimits}
\newcommand{\CS}{\operatorname{\mathcal S}\nolimits}
\newcommand{\CT}{\operatorname{\mathcal T}\nolimits}
\newcommand{\CU}{\operatorname{\mathcal U}\nolimits}
\newcommand{\CV}{\operatorname{\mathcal V}\nolimits}
\newcommand{\CX}{\operatorname{\mathcal X}\nolimits}
\newcommand{\CY}{\operatorname{\mathcal Y}\nolimits}
\newcommand{\kk}{\operatorname{\mathbb k}\nolimits}
\newcommand{\SE}{\operatorname{\mathscr E}\nolimits}

\newcommand{\bC}{\operatorname{\bf C}\nolimits}
\newcommand{\bA}{\operatorname{\bf A}\nolimits}
\newcommand{\Cpx}{\operatorname{\bf Cpx}\nolimits}
\newcommand{\Cpxb}{\operatorname{\Cpx^b}\nolimits}
\newcommand{\Cpxp}{\operatorname{\Cpx^+}\nolimits}
\newcommand{\Cpxm}{\operatorname{\Cpx^-}\nolimits}
\newcommand{\cpx}{\operatorname{\bf cpx}\nolimits}
\newcommand{\cpxb}{\operatorname{\cpx^b}\nolimits}
\newcommand{\cpxp}{\operatorname{\cpx^+}\nolimits}
\newcommand{\cpxm}{\operatorname{\cpx^-}\nolimits}
\newcommand{\cpxFL}{\operatorname{\bf cpxFL}\nolimits}
\newcommand{\cpxbFL}{\operatorname{\cpxFL^b}\nolimits}
\newcommand{\cpxpFL}{\operatorname{\cpxFL^+}\nolimits}
\newcommand{\cpxmFL}{\operatorname{\cpxFL^-}\nolimits}

\newcommand{\FA}{\operatorname{\mathfrak A}\nolimits}
\newcommand{\FX}{\operatorname{\mathfrak X}\nolimits}
\newcommand{\FY}{\operatorname{\mathfrak Y}\nolimits}
\newcommand{\FU}{\operatorname{\mathfrak U}\nolimits}
\newcommand{\FD}{\operatorname{\mathfrak P}\nolimits}

\newcommand{\FP}{\operatorname{\mathfrak P}\nolimits}

\newcommand{\Proj}{\operatorname{\bf Proj}\nolimits}

\newcommand{\proj}{\operatorname{proj}\nolimits}
\newcommand{\splt}{\operatorname{{\mathcal Splt}}\nolimits}

\newcommand{\Vpre}{\operatorname{{V\thyph}}\nolimits}
\newcommand{\Vsplt}{\operatorname{V\thyph\splt}\nolimits}

\newcommand{\Vts}{\operatorname{\Vpre\ts}\nolimits}

\newtheorem{lemma}{Lemma}[section]
\newtheorem{prop}[lemma]{Proposition}
\newtheorem{corollary}[lemma]{Corollary}
\newtheorem{theorem}[lemma]{Theorem}
\newtheorem{remark}[lemma]{Remark}
\newtheorem{definition}[lemma]{Definition}

\newtheorem{notation}[lemma]{Notation}

\date{\today}

\begin{document}

\thispagestyle{empty}
\title[the Green correspondence for
complexes]{Relatively projectivity and the Green correspondence for
complexes}
\author{Jon F Carlson}
\address{Department of Mathematics, University of Georgia,
Athens, GA 30602 USA}
\email{jfc@math.uga.edu}
\thanks{The first author was partially supported by
Simons Foundation Grant 054813-01.}
\author{Lizhong Wang}
\address{School of Mathematics,
Peking University, Beijing 100871, P.R.China}
\email{lwang@math.pku.edu.cn}
\author{Jiping Zhang}
\address{School of Mathematics,
Peking University, Beijing 100871, P.R.China}
\email{jzhang@pku.edu.cn}
\thanks{The second and third authors were supported by NSFC(11631001, 11871083)}

\begin{abstract}
We investigate a version of the Green correspondence for categories of
complexes, including homotopy categories and derived categories. The
correspondence is an equivalence between a category defined over a finite
group $G$ and the same for a subgroup $H$, often the normalizer of a
$p$-subgroup of $G$. We present a basic formula for deciding when categories
of modules or complexes have a Green correspondence and apply it to many
examples.  In several cases the equivalence is
an equivalence of triangulated categories, and in special cases it is an
equivalence of tensor triangulated categories.
\end{abstract}

\maketitle


\section{Introduction} \label{sec:intro}
The classical Green correspondence was defined by J. A. Green
\cite{green1} more than half a century ago.  It was one of several papers
by Green that changed the face of modular representation theory, with
an emphasis more on modules and maps rather than characters.
The correspondence expressed a relationship
given by induction and restriction, between
relative categories of modules of a finite group $G$ and
a subgroup $H$, where $H$ is usually taken as the normalizer
of the vertex of some module of interest. Green originally
stated it in terms of the representation rings
of the groups, {\it i. e.} as a correspondence of objects, with
little regards for the maps. In a later paper \cite{green2}, he
described it in terms of an equivalence of categories, with induction
and restriction being functors. Still, he
assumes that the modules are finitely generated and
the coefficient ring $k$ is either a field of
characteristic $p>0$ or a complete DVR whose residue ring is a
field of characteristic $p$. Basically, the assumption assures that
the categories have a Krull-Schmidt Theorem.
Many other extensions such as \cite{auslander-kleiner} \cite{harris},
to name just a couple,
also rely largely on the Krull-Schmidt property.
Later, a sweeping
generalization by Benson and Wheeler \cite{benson-wheeler},
proved equivalences of categories not only for infinitely generated
modules, but they also allowed the coefficients to be from any
commutative ring $k$ in which the order of the group is not invertible.

In this paper we carry the study a step further looking at a
variation on the Green correspondence for categories of complexes,
including homotopy categories and derived categories. We build somewhat
on the work \cite{wang-zhang} of the second and third authors.
The main issue is that we generalize the Benson-Wheeler
proof \cite{benson-wheeler} thereby providing axioms insuring
that induction and restriction give categorical equivalences.
The main theorem is presented in Section \ref{sec:functor}.
The primary difficulty in applying the axioms is
to show that  the categories
under consideration have the idempotent
completion property, {\it i. e.} they are Karoubian. This property is
the substitute for the Krull-Schmidt property, which is lacking in many
of the categories that we consider. Generally, the property holds
whenever a triangulated category has infinite direct sums.

The earlier sections of the paper are concerned with some explanation of
the numerous categories that we consider. In Section 2, we recall the
notion of an exact category and state some preliminaries. A main interest
is the quotient categories of complexes defined by relative projectivity,
relative to a $kG$-module. Included are categories of complexes of
$kG$-modules, and those complexes bounded or bounded above or below,
homotopy categories defined by term split sequences, or relative term
split sequences or relatively split sequences, sequences that split on
tensoring with a specific module. A primary goal
in Section 3 is
to determine the projective objects associated to the exact category
and to show that these are Frobenius categories.

In Section 4, we address the issue of idempotent completions. The
categories of complexes and their homotopy categories from the
previous sections are shown to be idempotent complete by usual
methods provided the coefficients are either Artinian or
infinitely generated modules are allowed. For complexes of
modules of finite composition length, it is proved that the
associated quotient categories defined by relative projective
objects are idempotent complete.

Section \ref{sec:acyclic} introduces the subcategories of acyclic
complexes, and the associated derived categories. Among other things,
it is shown that certain subcategories of acyclic objects defined by
relative projectivity are thick subcategories.

The main theorem on equivalences is Theorem \ref{thm:functorial}.
In Section \ref{sec:relprojtheory}, we recall
some of the standard constructions for
group representations, such as Frobenius Reciprocity and the Mackey
Theorem and show that these also hold for the categories of complexes
that we consider. This demonstrates that the categories satisfy
many of the conditions of the hypothesis of Theorem \ref{thm:functorial}.
The actual application of Theorem \ref{thm:functorial} takes place
in Section \ref{sec:applications}. 
The classical Green correspondence is extended to
the categories of complexes and their associated homotopy categories.
For derived categories, it is necessary to add an additional assumption.

In some cases, the equivalences associated to the
Green correspondence are equivalences of triangulated categories.
For example, suppose that $B$ is a
block of $kG$ with defect group $P$ and $b$ is its Brauer correspondent,
a block of $kH$ where $H$ is some subgroup that contains $N_G(P)$.
Then there is a triangulated  equivalence 
between the quotient category of homotopy
classes of bounded complexes of $B$-modules modulo 
$\FX$-projective complexes and
the same for $b$-modules modulo $\FX$-projectives complexes in $b$. 
Here $\FX$ is the collection of proper intersections 
of $P$ with its conjugates by elements not in $H$. See
Section \ref{sec:tri-equiv} for precise  details.
While the results
are mostly for the homotopy categories and derived categories of blocks,
they apply also to the stable category of modules. If $\FP$ is
the set of all proper intersections of a Sylow $p$-subgroup of $G$, then
the Green correspondence as above for $kG$-modules is an equivalence
of tensor triangulated categories.

For notation in this paper, let $G$ be a finite group
 and let $k$ be a commutative ring. Let $\Mod(kG)$
denote the category of all $kG$-modules and let $\bfmod(kG)$ be the
category of finitely generated $kG$-modules. These are tensor categories
in that, given two objects $M$ and $N$, there is tensor product
which is also an object in the category.
The $G$-action on  $M \otimes_k N$ is given by the diagonal
$g \mapsto g \otimes g$.
Throughout the paper, the
symbol $\otimes$ means $\otimes_k$ unless otherwise indicated.

In the first five sections of the paper, it seems helpful to make a clear
distinction between modules and complexes. So a complex $X^*$ is 
marked with the superscript ``*'', standing in place of a specific 
degree. This convention is relaxed in later sections where the 
notation presents other difficulties. 

We thank Dave Benson, Henning Krause and Paul Balmer and Jeremy Rickard
for helpful conversations. The first
author wishes to thank Peking University, Shangxi Normal University and
Southern University of Science and Technology for
their hospitality during visits
when much of this project was discussed.


\section{Basics on categories}   \label{sec:basiccats}
In this section we review a few basic categorical constructions that
are needed for the results of this paper. Most of this material is aimed
at stating and proving facts concerned with modules over group algebras,
and for this reason we make little attempt at great generality. For 
background information see the papers \cite{keller, keller2} or the 
books \cite{neeman, happel}.

A category is $k$-linear if all of its hom sets are $k$-modules and
composition of morphisms is $k$-linear. It implies that there is a forgetful
functor to the category of $k$-modules.
A $k$-linear category $\CC$ is hom-finite provided for
any two objects $M$ and $N$, $\Hom_{\CC}(M,N)$ has finite composition
length.

An additive category $\CC$ is a Krull-Schmidt
category if the objects satisfy a
Krull-Schmidt theorem, namely every object has a decomposition as a
direct sum of a finite number of indecomposable objects
and this decomposition
is unique up to ordering of the summands and isomorphisms of the
summands. This is equivalent to the condition that the endomorphism
ring of an indecomposable object in $\CC$ be a local ring.

Let $\Cpx(kG)$ be the category of complexes of $kG$-modules and
chain maps. Thus an object $X^*$ in $\Cpx(kG)$ is a complex
\[
\xymatrix{
\dots \ar[r] & X^n \ar[r]^{\partial_n} & X^{n+1} \ar[r]^{\partial_{n+1}}
& X^{n+2} \ar[r]^{\partial_{n+2}} & \dots
}
\]
of $kG$-modules and $kG$-module homomorphisms. For complexes $X^*$ and 
$Y^*$, a chain map $\mu: X^* \to Y^*$
is a sequence of maps $\{\mu_n: X^n \to Y^n\}$ such that $\partial^Y_n \mu_n
= \mu_{n+1}\partial^X_n$.

Let $\Cpxp(kG)$, $\Cpxm(kG)$ and $\Cpxb(kG)$ denote the full subcategories
of $\Cpx(kG)$ consisting of complexes that are bounded (respectively) above,
below or both above and below. All of these categories are $k$-linear.
Let $\cpx(kG) = \Cpx(\bfmod(kG))$ be the category of complexes of finitely
generated $kG$-modules, and let $\cpxp(kG)$, $\cpxm(kG)$ and $\cpxb(kG)$
be the bounded versions. Again these are all $k$-linear.
Note that if $k$ is a field, then also
these are tensor categories, except for $\cpx(kG)$.
The latter suffers from the fact that the tensor product of two complexes,
which are unbounded in both directions may not have finitely generated
terms even when the two complexes have finitely generated terms.

\begin{prop}\label{thm:KS1}
Suppose that $k$ is a field. The category $\cpxb(kG)$
of complexes of finitely generated $kG$-modules is a $k$-linear,
hom-finite, Krull-Schmidt category.
\end{prop}

\begin{proof}
That $\cpxb(kG)$ is $k$-linear and hom-finite is clear from the
definition. The fact that it is a Krull-Schmidt category follows from
a result of Atiyah \cite{atiyah} which states that any
abelian category satisfying a certain "bichain condition"
has the Krull-Schmidt property. However, it is easy to see that any
hom-finite category satisfies the condition.
\end{proof}

The notion of an exact category was first
introduced by Quillen in \cite{quillen},
and has been extensively developed since then.
Start with an additive categrory $\CC$. Let $\SE$ be a collection of
exact sequences of objects and maps in $\CC$. We assume that $\SE$
satisfies certain axioms. Among these are statements such as that
any exact sequence isomorphic to an element of $\SE$ is in $\SE$.
The first maps in the exact sequences are called admissible monomorphisms,
while the second maps are admissible epimorphisms. The composition of two
admissible monomorphisms is an admissible monomorphism, and similarly
for admissible epimorphisms. Also, admissible monomorphisms are
preserved by arbitrary pushouts while admissible epimorphisms are
preserved by arbitrary pull backs. See \cite{keller2} for precise details.

If $\SE$ is a collection of exact sequences in $\CC$ as above, then the
pair ($\CC$, $\SE$) is called an exact category.
An object $X\in \CC$ is called $\mathscr{E}$-projective
if for each exact sequence
\[
\xymatrix{
0 \ar[r] & A'\ar[r] &  A \ar[r] & A''\ar[r] & 0
}
\]
in $\mathscr{E}$, the sequence of groups
\[
\xymatrix{
0\ar[r]& \Hom(X,A')\ar[r]& \Hom(X,A)\ar[r]& \Hom(X,A'')\ar[r]&0
}
\] is exact. The notion of $\mathscr{E}$-injective is defined dually.

An exact category ($\CC$, $\SE$) has enough injectives,
if for each $A\in \CC$, there is an admissible monomorphism
$A \to I$ where $I$ is injective.
Also if for each $A\in \CC$, there is an admissible epimorphism
$P \to A$, where $P$ is $\mathscr{E}$-projective, then
we say that ($\CC$, $\SE$) has enough projectives. In such a category, we
denote the subcategory of projectives by $\SE$-$\Proj$.

An exact category ($\CC$, $\SE$) is a Frobenius category if
it has enough projectives and injectives and if the collections of
projective and injective objects coincide.

As an example, we recall that, if $k$ is a field, then
the projective objects and injective objects in $\Mod(kG)$ and in $\bfmod(kG)$
coincide. Then the category $\Cpx(kG)$, where we assume that
the collection of sequence $\SE$ to be all exact sequences,
is a Frobenius category. The
projectives in this category are the complexes of
projective modules that are both split and exact.
These are complexes $X^*$ which are direct sums of
two term exact complexes having the form $0 \to P \to P \to 0$ where $P$
is a projective $kG$-module and nonzero terms occur in degree $n$ and
$n+1$ for some $n$. That is, the $\SE$-projective objects are the
null-homotopic complexes of projective modules.
It is easy to see that these complexes are also
injective.

Suppose that ($\CC$, $\SE$) is a Frobenius category.
Associated to ($\CC$, $\SE$), there is a stable
category or quotient category which we denote $\CD = {\CC}/({\SE}$-$\Proj$).
The objects in the category coincide with the objects in $\CC$. For
$X$ and $Y$ objects in $\CC$ the morphisms are given by
\[
\Hom_{\CD}(X,Y) = \frac{\Hom_{\CC}(X,Y)}{\PHom_{\CC}^{\SE}(X,Y)}
\]
where $\PHom_{\CC}^{\SE}$ are the homomorphisms that factor through an
$\SE$-projective module. This is a triangulated
category, the triangles corresponding to sequences in $\SE$.
The method is briefly described as follows. See \cite{happel} for
more details.

If $M$ is an object in $\CC$, there is a sequence in $\SE$
\[
\xymatrix{
0 \ar[r] & M \ar[r] &  I_{\SE} \ar[r] & \Omega^{-1}_{\SE}(M) \ar[r] & 0
}
\]
where $I_{\SE} = I_{\SE}(M)$ is an $\SE$-injective object that
serves the purpose of a relative injective hull.
It defines (up to isomorphism in $\CD$) the object $\Omega^{-1}_{\SE}(M)$.
The operator $\Omega^{-1}_{\SE}$ is a functor on the stable
category $\CD$.  Then a given morphism, $\varphi: M \to N$, is fit
into a triangle by means of the pushout diagram
\[
\xymatrix{
0 \ar[r] & M \ar[r] \ar[d]^\varphi & I \ar[r] \ar[d] &
\Omega^{-1}_{\SE}(M) \ar[r] \ar@{=}[d] & 0 \\
0 \ar[r] & N \ar[r]^\beta & U \ar[r]^{\gamma \quad} &
\Omega^{-1}_{\SE}(M) \ar[r] & 0
}
\]
Here $I$ is the $\SE$-injective hull of $M$ as above, and $U$ is
the pushout in the left square. Then a triangle containing the
class of $\varphi$
is given as
\[
\xymatrix{
M \ar[r]^\varphi & N \ar[r]^\beta & U \ar[r]^\gamma & \Omega^{-1}_{\SE}(M)
}
\]

An important example is the collection $\CT = \tsseq(\Cpx(kG))$
of term split sequences of
complexes in $\Cpx(kG)$ \cite{grime}. 
A sequence $0 \to A^* \to B^* \to C^* \to 0$ of
objects in $\Cpxb(kG)$ is term split if for every degree $d$ the sequence
of terms $0 \to A^d \to B^d \to C^d \to 0$ is a split sequence of
$kG$-modules. The $\tsseq$-projective objects in the exact category
$(\Cpx(kG), \CT)$ are the split exact complexes. A complex $X^*$ is
split exact if for every $d$ there is a map $s_d: X^{d+1} \to X^d$
such that $\partial_{d-1}s_{d-1} + s_d\partial_d:X^d \to X^d$ is the
identity homomorphism. Such a complex is a direct sum of two-term complexes
$0 \to X \to X \to 0$ where the nontrivial map
is the identity. In other words, $\CT$-$\Proj$ is the collection
of complexes that are homotopic to the zero complex.

Let $\CK(kG) = \CK(\Cpx(kG))$,
denote the homotopy category of complexes of $kG$-modules
and homotopy classes of chain maps.  It is a triangulated category
and its translation functor
$\Omega^{-1}_{\tsseq}$ is the shift functor that takes $X^*$ to
$X[1]^*$ where $X[1]^n = X^{n+1}$ for all $n$ and the boundary maps
are all multiplied by $-1$. If $f: Y^* \to X^*$
is a chain map then the third object in the triangle of $f$ is
isomorphic to the usual mapping cone $M(f)$. Here $M(f)^n =
X^n \oplus Y^{n+1}$ and the boundary map $M(f)^n \to M(f)^{n+1}$ is
given by $\partial(x,y) = (\partial(x)+f(y), -\partial(y))$.
We let $\CK^*(kG) = \CK(\Cpx^*(kG))$ for $* = +, -$ or
$b$ be the full subcategory of
$\CK(kG)$ with objects in $\Cpx^*(kG)$.

The following is well known.

\begin{prop} \label{prop:homotop}
The quotient category $\Cpx(kG)/\CT$-$\Proj = \CK(kG)$,
by the projectives of the set of term split sequences,
$\CT = \tsseq(\Cpx(kG))$, is the homotopy category
$\CK(kG)$ of complexes of $kG$-modules and homotopy classes of
maps. Likewise, for $* = +, -$ or $b$,
$\CK(\Cpx^*(kG)) = \Cpx^*(kG)/\tsseq(\Cpx^*(kG))$-$\Proj$.
We similarly denote the homotopy categories of complexes of
finitely generated
modules $\CK(\cpx(kG))$ and $\CK(\cpx^*(kG)).$
These are tensor triangulated category (except for
$\CK(\cpx(kG))$ which has no tensor). The exact categories
$(\Cpx^*(kG), \CT)$ and $(\cpx^*(kG), \tsseq(\cpx^*(kG)))$ are
Frobenius categories.
\end{prop}

\begin{proof}
The fact that $\CK(kG)$ is triangulated follows from standard arguments
and is well known. It has a tensor structure because the class $\tsseq$
is closed under arbitrary tensors. That is, the sequence of objects in
a term split sequence splits as a sequence of $k$-modules. So if
$0 \to A^* \to B^* \to C^* \to 0$ is a term split sequence of $kG$-complexes
then so is $0 \to A^* \otimes X^* \to B^* \otimes X^* \to
C^* \otimes X^* \to 0$ for any complex $X^*$ in $\Cpx(kG)$. It does
not matter that tensoring with any $X^i$ might not be an exact functor.
The same holds for the considered subcategories. That is, for example,
the tensor of a term split sequence in $\cpxp(kG)$ with any object in
$\cpxp(kG)$ is again a term split sequence in $\cpxp(kG)$.
\end{proof}

\begin{prop} \label{prop:KS2}
Suppose that $k$ is a field.
The category $\CK(\cpx^b(kG))$ of bounded complexes of finitely generated
$kG$-modules and homotopy classes of maps is a Krull-Schmidt category.
\end{prop}

\begin{proof}
As above it can be seen that $\CK(\cpx^b(kG))$ is the stable or quotient
category of the exact category ($\cpx^b(kG), \CT$), where
$\CT$ is the collection of term split sequences of bounded complexes
of finitely generated $kG$-modules. Now, $\CK^b(\Cpx(kG))$
is a hom-finite, Krull-Schmidt category. Thus,
the Krull-Schmidt property is a consequence of the observation
that the endomorphism
ring of an indecomposable object is a quotient of a
finite dimensional local ring, and hence
remains a local ring. That $\CK^b(\Cpx(kG))$ is a Frobenius category follows
from the arguments given below.
\end{proof}	

Let $\CS(\cpxb(kG))$ denote the full subcategory of $\cpxb(kG)$
consisting of all complexes of $kG$-modules that are free and split
on restriction to $k$. That is,
the restriction to $k$ of such a complex is a finite direct sum of
complexes having the form either
\[
\xymatrix@-.4pc{
\dots \ar[r] & 0 \ar[r] & k \ar[r]^{\Id} & k \ar[r] & 0 \ar[r] & \dots
\quad \text{or} \quad
\dots \ar[r] & 0 \ar[r] & k \ar[r] & 0 \ar[r] & \dots
}
\]
In particular, it is direct sum of one- and two-term sequences, and the terms
are $k$-isomorphic to $k$.

For $X^*$ in $\CS(\cpxb(kG))$, let $(X^\#)^* = \Hom_k(X^*, k)$ be its
$k$-dual. It is the complex $(X^\#)^d = \Hom_k(X^{-d},k)$ and with
boundary map being the dual of the boundary map for $X^*$, adjusted by
a sign.  That is, the boundary map $\delta_d: (X^\#)^d \to (X^\#)^{d+1}$
is given by $\delta_d = (-1)^{d+1}\partial_{-d-1}^\#$, where
$\partial_{-d-1}^\#:(X^{d+1})^\# \to (X^{d})^\#$ is the ordinary dual:
$(\partial_{-d-1}^\#(\lambda))(x) = \lambda(\partial_{-d-1}(x))$ for
$\lambda \in (X^{d+1})^\#$, and $x \in X^d$.
For $Y^*$ any element of $\Cpx(kG)$, there is an isomorphism
$\bhom_k^*(X, Y) \cong (X^\#)^* \otimes Y^*$.    
Here $\bhom^j_k(X,Y)$ is the collection $\sum \Hom_k(X^i, Y^{i+j})$. The 
aggregrate $\bhom_k^*(X, Y)$ 
is a complex. An element in $\bhom^j_k(X,Y)$ is an indexed sequence of 
maps $\{f_i\}$ where $f_i: X^i \to Y^{i+j}$ is $k$-homomorphism. The boundary
map on the complex takes $f_i$ to 
$\partial(f_i) = (-1)^i(\partial_Y\circ f_i - f_i\circ\partial_X)$. 

There is the usual adjointness
\[
\Hom_{\Cpx(kG)}(V^* \otimes X^*, U^*) \cong
\Hom_{\Cpx(kG)}(V^*,  (X^\#)^* \otimes U^*),
\]
for any complexes $U^*$ and $V^*$.

Suppose that $U$ is a finitely generated $kG$-module. There is a trace map
$\Tr = \Tr_U: U^\# \otimes U \to k$
given by $\Tr(\lambda \otimes u) = \lambda(u)$ for $\lambda$ in
$U^\#$ and $u$ in $U$. Viewed from the isomorphism  $U^\# \otimes U \cong
\Hom_k(U,U)$, it becomes the usual trace map on matrices.
With a sign convention, it extends to a chain map on complexes.
Let $\kk^*$ be the complex having only one nonzero term which is in degree
zero and is equal to $k$. Then for any $U^*$ in $\CS(\cpxb(KG))$ there is a
trace map $\Tr: (U^*)^\# \otimes U^* \to \kk^*$, which in degree zero
$\Tr: ((U^*)^\# \otimes U^*)_0 \to k$ is the super trace map.
That is, on $(U^d)^\# \otimes U^d$, the super trace is $(-1)^d\Tr_{U^d}$.

Likewise for $U$ a finitely generated $kG$-module,
there is a unit homomorphism
$\iota = \iota_U: k \to U^{\#} \otimes U$,
which sends 1 to the identity homomorphism
$\Id_U \in \Hom_{kG}(U,U)$. Here $\Id_U = \sum \lambda_i \otimes v_i$
where $\{\lambda_i\}$ and $\{v_i\}$ are dual bases of $U^\#$ and $U$,
respectively. This map is dual to the trace map. For $U^*$ a complex in
$\CS(\cpxb(kG))$ there is also a unit homomorphism
$\iota: \kk^* \to (U^*)^\# \otimes U^*$ which is dual to the trace map.

It can be calculated that the composition
\[
\xymatrix@+1pc{
U^* \ar[r]^{\iota \otimes 1 \quad\qquad} &
(U^{\#})^* \otimes U^* \otimes U^* \ar[r]^{\cong} &
U^* \otimes (U^{\#})^* \otimes U^* \ar[r]^{\qquad\quad 1 \otimes \Tr} & U^*
}
\]
is the identity map.

\begin{lemma} \label{lem:chain1}
Assume that $U^* \in \CS(\cpxb(KG))$.
The maps $\Tr_U$ and $\iota_U$ are chain maps.
Moreover, $U^*$ is a direct summand of $U^* \otimes (U^*)^\# \otimes U^*$.
\end{lemma}

\begin{proof}
By hypothesis, $U^*$ is a bounded complex of finitely generated
$kG$-modules whose restriction to $k$ is both free and split.
\[
\xymatrix{
\dots \ar[r] & U^{-1} \ar[r]^\partial & U^0 \ar[r]^\partial &
U^{1} \ar[r]  &  \dots
}
\]
Its dual has boundary map $\delta: (U^i)^\# \to (U^{i-1})^\#$ is given
by $\delta(\lambda) = (-1)^{i-1}\lambda \circ \partial$.
Then the tensor product has the form
\[
\xymatrix{
(U^\# \otimes U)^*: & \dots \ar[r] & (U^\# \otimes U)^{-1} \ar[r] &
(U^\# \otimes U)^0 \ar[r]^\partial &
(U^* \otimes U)^{1} \ar[r]  &  \dots
}
\]
Let $\kk^*$ be the complex with $\kk^0 = k$ and all other terms zero.
To check that the super trace $\Tr: (U^\# \otimes U)^* \to \kk^*$ is
a chain map. we need only check that $\Tr\partial: (U^\# \otimes U)^{-1}
\to \kk^0 \cong k$ is the zero map. For this, choose $\lambda_i \in
(U^i)^\#$ and $x_{i-1} \in U^{i-1}$. Then
\begin{align*}
\Tr\partial(\lambda_{i} \otimes u_{i-1}) & =
\Tr(\delta(\lambda_{i}) \otimes u_{i-1} + (-1)^{i}\lambda_{i} \otimes
\partial(u_{i-1})) \\
     & = (-1)^{i-1}\delta(\lambda_{i})(u_{i-1}) +
      (-1)^{i}(-1)^{i-1}\lambda_{i}\partial(u_{i-1}) \\
     & = (1 + -1) \lambda_{i-1}\partial(u_{i-1}) = 0.
\end{align*}
The unit map is the dual of the (super) trace map and hence is also
a chain map.

For the second statement, we note that the composition
of $I \otimes 1$ with $1 \otimes \Tr$ is the identity of $U^* \cong
\kk^* \otimes U^* \cong U^* \otimes \kk^*$
(see \cite{auslander-carlson} for the module version).
Because $\cpxb(kG)$ is an abelian
category, this gives us a direct sum splitting.
\end{proof}

We say that collection $\SE$ of exact sequences
in a tensor subcategory $\bC$ of $\Cpx(kG)$ is closed
under arbitrary tensor products, if
whenever a sequence $0 \to A^* \to B^* \to C^* \to 0$ is in $\SE$, then
so also is $0 \to A^* \otimes X^* \to B^* \otimes X^* \to
C^* \otimes X^* \to 0$ for any complex $X^*$ in $\bC.$

With these notions in mind we can prove the following.

\begin{theorem} \label{thm:tensor}
Suppose that $\SE$ is a collection of sequences in $\Cpx(kG)$ such that
($\Cpx(kG)$, $\SE$) is an exact category. We assume the following.
\begin{itemize}
\item[a.] $\SE$ is closed under arbitrary tensor products.
\item[b.] The trivial complex $\kk^*$ has a projective cover
$\psi: P(\kk)^* \to \kk^*$ relative to $\SE$ in $\CS(\cpxb(KG))$.
In particular, we have an exact sequence in $\SE$
\[
\xymatrix{
E_{\kk} : & 0 \ar[r] & \Omega_{\SE}(\kk)^* \ar[r] &
P(\kk)^* \ar[r]^\psi & \kk^* \ar[r] & 0
}
\]
where $P(\kk)^*$ is in $\SE$-$\Proj$ and in $\CS(\cpxb(KG))$. That is, each
$P(\kk)^n$ is a free $k$-module of finite rank.
\item[c.] For any object $X^*$, $X^* \otimes P(\kk)^*$ is $\SE$-projective.
\item[d.] The dual sequence to $E_{\kk}$ is in $\SE$ and $(P(\kk)^*)^\#$
is $\SE$-injective.
\item[e.] For any object $X^*$, $X^* \otimes (P(\kk)^*)^\#$ is $\SE$-injective.
\end{itemize}
Then we have the following.
\begin{enumerate}
\item $(P(\kk)^*)^\#$ is $\SE$-projective.
\item An object is  $\SE$-projective if
and only if it is a direct summand of an object having
the form $X^* \otimes P(\kk)^*$ for some complex $X^*$,
\item Every $\SE$-injective complex is a direct summand of an object having
the form $X^* \otimes (P(\kk)^*)^\#$ for some complex $X^*$,
\item For any complex $X^*$ and any $n$, $\Omega_{\SE}^n(\kk^*) \otimes X^*
\cong \Omega_{\SE}^n(X^*) \oplus Y^*$
for some $\SE$-projective complex $Y^*$.
\item The exact category ($\Cpxb(kG)$, $\SE$) has enough projectives and
injectives, and is a Frobenius category.
\end{enumerate}
\end{theorem}

\begin{proof}
The first statement is a direct consequence of assumption (c) and Lemma
\ref{lem:chain1}, since $(P(\kk)^*)^\#$ is a direct summand of
$(P(\kk)^*)^\# \otimes P(\kk)^* \otimes (P(\kk)^*)^\#$.
For the second statement, suppose that $X^*$ is $\SE$-projective.
Then the sequence $\Hom(X^*, E_{\kk} \otimes X^*)$ is exact, since
$E_{\kk} \otimes X^*$ is in $\SE$. However, this implies that
$X^*$ is a direct summand of $P(\kk)^* \otimes X^*$. The converse is
statement (c).

The dual argument concludes that if $X^*$ is $\SE$-injective then it is a
direct summand of $(P(\kk)^*)^\# \otimes X^*$ and hence it is projective.
Likewise, $\SE$-projective objects are also relatively injective.
The last two statements use the facts that the sequence $E_{\kk} \otimes X^*$
is a relative projective cover of $X^*$ and $E_{\kk}^\# \otimes X^*$
is a relative injective hull of $X^*$.
\end{proof}

Note that if $\SE$ is the collection of term split sequences, then
$P(\kk)^*$ is the two term complex with the nonzero terms in degrees
0 and 1, as in the diagram
\[
\xymatrix{
P(\kk)^*: \ar[d]^{\psi}& \dots \ar[r] & 0 \ar[r] \ar[d]&
k \ar[r] \ar[d]^{\psi_0}
& k \ar[r] \ar[d] & 0 \ar[r] & \dots \\
\kk^* : & \dots \ar[r] & 0 \ar[r] & k \ar[r] &  0 \ar[r] & \dots
}
\]
In this case $P(\kk)^*$ satisfies all of the conditions of the above theorem.

\begin{theorem}
Let $\bC$ denote one of the categories $\Cpx^*(kG)$, $\cpx(kG)$ or
$\cpx^*(kG)$ for $* = +, -$ or $b$.
Suppose that $\SE$ is a collection of sequences of
objects in $\SE$ such that
($\bC$, $\SE$) is an exact category, and that $\CE$ satisfies the hypotheses
of Theorem \ref{thm:tensor}. Then the conclusions of that theorem also
hold for ($\bC$, $\SE$).
\end{theorem}

\begin{proof}
The proof is the same as for Theorem \ref{thm:tensor}. It is only necessary to
notice such thing as the fact that $
P(\kk)^* \otimes X^*$ is in $\bC$ whenever $X^*$ is in $\bC$.
\end{proof}


\section{Relative projectivity}\label{sec:relproj}
In this section we present the basics of relative projectivity, relative to
a module and to a collection of subgroups. The ideas take place in the
context of $kG$-modules where $G$ is a finite group and $k$ is a commutative
ring such that the order of $G$ is not invertible in $k$. For some background
see the paper \cite{carlson-peng-wheeler}.
Throughout the section, the symbol $V$ denotes a $kG$-module
that, as a module over $k$, is free and has finite rank.

\begin{definition} \label{defi:relproj} \cite{okuyama}
Suppose that $V$ is a finitely generated $kG$-module. A $kG$-module $M$ is
said to be relatively $V$-projective provided $M$ is a direct summand of
$V \otimes X$ for some $kG$-module $X$. A complex $X^*$ in $\Cpx(kG)$
is said to be $V$-projective if it is a direct summand of $V \otimes Y^*$
for some complex $Y^*$ in $\Cpx(kG)$.
\end{definition}

The full  subcategory of $V$-projective complexes is
denoted $V$-$\Proj(\Cpx(kG))$ or just $V$-$\Proj$ if there is
no confusion. It is closed under direct sums and summands,
but not under extensions. It is also closed under tensor products with
arbitrary modules and complexes, by associativity.

Suppose that $\CH$ is a collection of subgroups 
of $G$. We say that $X^* \in \Cpx(kG)$ is relatively $\CH$-projective
if it is $V$-projective for $V = \sum_{H \in \CH} k_H^{\uparrow G}$. In
other words, by Frobenius reciprocity (see Theorem \ref{thm:std-reptheory}), 
$X^*$ is $\CH$-projective if it
is a summand of a direct sum of complexes of modules induced from
$\Cpx(kH)$ for $H \in \CH$.

\begin{definition}\label{def:vsplit}
An exact sequence of objects $E: 0 \to A^* \to B^* \to C^* \to 0$ in
$\Cpx(kG)$ is $V$-split if the tensor product $V \otimes E$ is split.
It is $V$-term split if $V \otimes E$ is term split. It is term+$V$-split
if it is both term split and also $V$-split.
\end{definition}

We note one fact that will be of some use in later sections. Its proof is
simply that if $E$ is a sequence of $k$-modules, the $k\otimes_k E$ is split
if and only if $E$ is split.

\begin{lemma}
Suppose that $U$ is free as a module over $k$. Then any $U$-split sequence
is split as a sequence of $k$-modules.
\end{lemma}

The collections of $V$-split sequences,
$V$-term split sequences and term+$V$-split sequences are denoted
$\Vsplt$-$\seq$, $V$-$\ts$-$\seq$ and $\ts$+$\Vsplt$-$\seq$, respectively.
All of these collections are closed under arbitrary tensors
since $V$ is free as a $k$-module. The next
task is to identify the projective objects relative to the collections
and to show that the exact categories are Frobenius categories. For
this, it is helpful to have some additional information on the
relative projectivity.

Note here that, if $k$ is a field of characteristic $p$ dividing the
order of $G$ and if $V$ is an absolutely indecomposable,
finitely generated $kG$-module, then
the trace map $V^\# \otimes V \to k$ is split if and
only if $p$ does not divide the dimension of $V$ (see \cite{benson-carlson}).
That is, the trace map is split
if and only if $\Tr(\Id_V) \neq 0$. In the case that $p$ does not divide
the dimension of $V$ or of any direct summand of $V$, then we have
that $k$ is $V$-projective and hence every $kG$-module is $V$-projective.
For the rest of this paper we avoid this situation.

We introduce the following variation on the homotopy category.
As we use the notation several times we here give it a label.
For the remainder of the section assume the following notation. 

\begin{notation} \label{nota:cpx}
Let $\bC$ denote any one of the categories
$\Cpx(kG)$, $\cpx(kG)$, $\Cpx^*(kG)$, or $\cpx^*(kG)$
for $* = +, -$ or $b$. If the group is in question, denote
the category $C_G$.
\end{notation}

\begin{prop} \label{prop:Vsplitproj}
Assume that $V$ is free of finite rank as a $k$-module.
Let $\SE = \Vsplt$-$\seq$ the collection of $V$-split
sequences in $\bC$. Then the  quotient category
$\CK_{\Vsplt}(\bC) =$ $\bC/\SE$-$\Proj$ is a Frobenius category.
The projective objects form the set
$\SE$-$\Proj$ which consists of all direct summands of objects having the
form $X^* \otimes V^{\#} \otimes V$ for $X^*$ in $\bC(kG)$. In
addition, if $k$ is a field then, the category
$\CK_{\Vsplt}(\cpxb(kG))$ is a Krull-Schmidt category.
\end{prop}

\begin{proof}
If $k$ is a field, then the category $\CK_{\Vsplt}(\cpxb(kG))$,
is a quotient of a Krull-Schmidt category.
That is, the endomorphism ring of any indecomposable object is
the quotient of a finite dimensional local $k$-algebra and hence is a local
ring. 

The rest of the proof follows from Theorem \ref{thm:tensor}, once we determine
the projective objects. Let $\CV^*$ be the complex with only one
nonzero term, which is $V^\# \otimes V$ in degree zero. Then there is an
exact sequence
\[
\xymatrix{
0 \ar[r] & W^*  \ar[r] & \CV^* \ar[r]^{\mu_*} & \kk^* \ar[r] & 0
}
\]
where $\mu_0 = \Tr: V^\# \otimes V \to k$ is the trace map in degree zero.
This sequence is $V$-split by Lemma \ref{lem:chain1},
and the middle term is $V$-projective,
implying that the middle term is in $\SE$-$\Proj$.
Consequently, it is the sequence of a relative projective cover of
$\kk^*$ plus, perhaps, a sequence having the form
$0 \to U^* \cong U^* \to 0$. Here, $U^*$
is a complex with only one nonzero term that is a direct summand
of $V^\# \otimes V$, which is $V$-projective. In a similar fashion
we see that the dual sequence is an injective hull of $\kk^*$. The fact that
$\SE$ is closed under arbitrary tensor products is obvious.
Observe that $\CV^* \otimes X^* \cong V^\# \otimes V \otimes X^*$. Thus
$\CV^* \otimes X^*$ is $\SE$-projective since for any sequence $E$
in $\SE$, we have that
$\Hom_{kG}(\CV^* \otimes X^*, E) \cong
\Hom_{kG}(X^*, V^\# \otimes V \otimes E)$ which is exact. The proofs
of statements  about duals in Theorem \ref{thm:tensor} are similar.
Hence the hypotheses of that theorem are all satisfied and the proof
is complete.
\end{proof}

\begin{prop} \label{prop:VTsplitproj}
Assume that $V$ is free of finite rank as a $k$-module.
Let $\SE = V$-$\ts$-$\seq$ in $\bC$. The quotient category
$\CK_{V\thyph\ts}(\bC) =$ $\bC/\SE$-$\proj$
is a Frobenius category. The projective objects form the collection
$\SE$-$\Proj$ consisting of all objects that are both split
and contractible complexes of $V$-projective
modules. That is $\SE$-$\Proj$ contains all complexes
having the form $\dots 0 \to W \to W \to 0 \dots$ where $W$ is
$V$-projective, the nonzero terms occur in consecutive degrees
and the boundary map is the identity. It consists of all
direct sums of such sequences when such a direct sum is in $\bC$.
If $k$ is a field, then
the category $\CK_{V\thyph\ts}(\cpxb(kG))$
is a Krull-Schmidt category.
\end{prop}

\begin{proof}
The proof is similar to the previous proof except for the issue of
the projective cover of the trivial complex.
Again, $\SE$ is closed under arbitrary tensor products.
Let $\CV^*$ be the complex with all zero terms except
for $V^\# \otimes V$ in degrees zero and one. Then we have a map of
complexes
\[
\xymatrix{
\CV^*: \ar[d]^{\mu} & \dots 0 \ar[r] &
V^\# \otimes V \ar[r]^{\Id_V} \ar[d]^{\Tr} &
V^\# \otimes V \ar[r] \ar[d] & 0 \ar[r] & \dots \\
\kk^*: & \dots 0 \ar[r] & k \ar[r] & 0 \ar[r] & 0 \ar[r] & \dots
}
\]
If $U^*$ is the kernel of $\mu$, then the sequence
\[
\xymatrix{
0 \ar[r] & U^* \ar[r] & \CV^* \ar[r]^{\mu} & \kk^* \ar[r] & 0
}
\]
is in $\SE$. Hence, all $\SE$-projective complexes are directs summands of
complexes having the form $\CV^* \otimes X^*$ for some complex $X^*$. It is
now an easy exercise to show that they have the stated form. The rest follows
from Theorem \ref{thm:tensor}.
\end{proof}

\begin{prop} \label{prop:V+Tsplitproj}
Assume that $V$ is free of finite rank as a $k$-module.
Let $\SE = \ts$+$\Vsplt$-$\seq$. The quotient category
$\CK_{\ts+\Vsplt}(\bC) =$ $\bC/\SE$-$\Proj$
is Frobenius category. The projective objects is the set
$\SE$-$\Proj$ consisting of all complexes that can be written as a
direct sum of an object in $\ts$-$\Proj$ and $\Vsplt$-$\Proj$.
\end{prop}

\begin{proof}
Suppose that $\hat{\kk}^*$ is the complex with $k$ in degrees 0 and 1,
the map between them being the identity and all other terms equal to zero.
Let $\CV$ be the complex with $V^\# \otimes V$ in degree zero and
all other terms equal to zero. Then we have diagram
\[
\xymatrix{
\CV^* \oplus \ \hat{\kk}^* : \ar[d]^{\mu} & \dots 0 \ar[r] &
(V^\# \otimes V) \oplus k \ar[r]^{\qquad (0,\Id_k)} \ar[d]^{(\Tr,Id_k)} &
k \ar[r] \ar[d] & 0 \ar[r] & \dots \\
\kk^*: & \dots 0 \ar[r] & k \ar[r] & 0 \ar[r] & 0 \ar[r] & \dots
}
\]
The chain map $\mu$ is both $V$-split and term split, and the complex
$\CV^* \oplus \hat{\kk}^*$ is in $\ts$+$\Vsplt$-$\Proj$. Thus, the projective
objects are as stated. The rest of the proof proceeds as before.
\end{proof}


\section{Finite length complexes and idempotent completions} \label{sec:FL}
In this section we consider which of the categories that we have discussed
have idempotent completions, as well as investigate idempotent completion
property for homotopy categories of complexes of modules of finite length.
This discussion is crucial to the application of the Green correspondence
as formulated in Section \ref{sec:functor}.
The idempotent complete
property is a weak substitute for the Krull-Schmidt property in
categories. Let $X$ be an object in an additive
category $\CC$. An idempotent $e \in \Hom_{\CC}(X,X)$ is said to be split
provided $X$ is a direct sum $X = X^\prime \oplus {X^\prime}{}^\prime$ in
such a way that the restriction of $e$ to $X^\prime$ is the identity map
and to $X^{\prime}{}^\prime$ is zero. The category $\CC$ is idempotent
complete provided every idempotent splits in $\CC$. An abelian category
is idempotent complete simply because it has kernels in cokernels.
A triangulated category that has countable
direct sums is idempotent complete \cite{bokstedt-neeman}.  Of course,
any Krull-Schmidt category has idempotent completions.
With this information we can prove the following.

\begin{prop} \label{prop:idemcomp}
Assume that $k$ is a commutative ring and $G$ is a finite group.
The categories $\bC = \bC(kG)$ where $\bC$
is one $\Cpx$, $\Cpxp$, $\Cpxm$ or $\Cpxb$
are idempotent complete. So also are their homotopy categories
$\CK(\bC)$, $\CK_{\Vsplt}(\bC)$, $\CK_{V\thyph\ts}(\bC)$
and $\CK_{\ts+\Vsplt}(\bC)$ where
$V$ is a $kG$-module that is free of finite rank as a $k$-module.
\end{prop}

\begin{proof}
The categories denoted $\bC$ are abelian and therefore idempotent complete.
For the homotopy categories, the only problem is that
$\Cpxp(kG)$, $\Cpxm(kG)$ and $\Cpxb(KG)$ do not have arbitrary
direct sums and hence neither do their homotopy categories.
That is, the direct sum of an infinite number of bounded complexes
may no longer by bounded. However, the proof
\cite{bokstedt-neeman} of the existence of
a splitting for an idempotent requires a homotopy limit construction
that, in turn, requires only a direct sum of a countable number of
copies of the complex on which the idempotent is defined.
Such a direct sum exists in these category, and hence the splitting
of any idempotent exists.
\end{proof}

For the homotopy categories of complexes of finitely generated modules
the problem is more difficult. The categories $\cpx(kG)$, $\cpxp(kG)$,
$\cpxm(kG)$ and $\cpxb(kG)$ have idempotent completions because they
are abelian. If $k$ is a field, then
the same conclusion holds for the homotopy categories
$\CK(\cpxb(kG))$, $\CK_{\Vsplt}(\cpxb(kG))$,
$\CK_{\ts+\Vsplt}(\cpxb(kG))$ and $\CK_{V\thyph\ts}(\cpxb(kG))$
because they are Krull-Schmidt categories. These are part of a more
general collection of categories that are idempotent complete.

Let $\cpxFL(kG)$ denote the subcategory of $\cpx(kG)$, consisting of
all complexes with the property that every term in the complex has
finite composition length. That is, every term has a
composition series of finite length in which the quotients
of successive terms in the series are 
simple $kG$-modules. Similarly let $\cpxpFL(kG)$, $\cpxmFL(kG)$ and
$\cpxbFL(kG)$ be the categories of complexes of finite length modules
that are bounded above or below or just bounded.

Notice that $\cpxFL(kG)$ is a full subcategory of $\cpx(kG)$ and that
a complex $X^*$ in $\cpx(kG)$ is in $\cpxFL(kG)$ if and only if
every term $X^i$ has
finite length as a module over $k$. If $V$ is a $kG$-module that is free
of finite rank over $k$ then $V \otimes X^*$ is in $\cpxFL(kG)$ if and
only if $X^*$ is in $\cpxFL(kG)$. Moreover, $X^*$ is projective relative
to $V$ if and only if it is a direct summand of $X^* \otimes V \otimes V^\#$.

We are indebted to Jeremy Rickard for most of the proof of the following.

\begin{theorem} \label{thm:idemfinleng}
Let $\bC$ denote any of the $FL$ categories of complexes: $\cpxFL(kG)$,
$\cpxpFL(kG)$, $\cpxmFL(kG)$ or $\cpxbFL(kG)$. Then $\bC$ has idempotent
completions and so does any of the homotopy categories
$\CK(\bC)$, $\CK_{\Vsplt}(\bC)$, $\CK_{\ts+\Vsplt}(\bC)$
and $\CK_{V\thyph\ts}(\bC)$. We are assuming here that $V$ is a $kG$-module
that is free of finite rank over $k$.
\end{theorem}

\begin{proof}
Note that the categories of complexes are abelian and hence also
idempotent complete. So we may assume that we are in one of the
homotopy categories which we call $\CK$.
Suppose that $X^*$ is an object in $\CK$ and that
$e: X^* \to X^*$ is a chain map such that $e^2 = e$ in $\CK$. That is,
$e^2$ is homotopic to $e$ in the sense that $e^2-e$ factors through the
appropriate relatively projective object. For each $i$, 
we have nested sequences of submodules
\[
e_iX^i \supseteq e_i^2X^i \supseteq e_i^3X^i \supseteq \dots
\]
and
\[
\Ker(e_i) \subseteq \Ker(e_i^2) \subseteq \Ker(e_i^3) \subseteq \dots
\]
Because $X^i$ has finite composition length, both sequences stabilize.
Let $Y^i$ and $Z^i$ be the limit modules. Then we have that for $n$
sufficiently large, depending on $i$,
$e_i^nY^i = Y^i$ and $e_i^nZ^i = \{0\}$. The
boundary map commutes with $e$, and hence, $X^* = Y^* \oplus Z^*$.
On the complex $Y^*$, $e$ acts as $e_Y$, an isomorphism, while on $Z^*$
it acts as $e_Z$ which is nilpotent in every degree.

Next we should note that both $e_Y$ and $e_Z$ are idempotent. That is,
the homotopy can be made to respect the decomposition of $X^*$ into a
direct sum. For example, suppose we are in the category
$\CK_{\ts+\Vsplt}(\bC)$. Then we have that
$e -e^2 = \partial d + d\partial + f$ where $d$
is a homotopy in the usual sense, $d_j:X^j \to X^{j-1}$,
and $f$ factors through a $V$-projective complex, which we can assume
to be $X^* \otimes V \otimes V^\#$. Put in matrix form we have that,
on $X^i$,
\[
\begin{bmatrix} e_Y - e_Y^2 & 0 \\
                  0   &  e_Z - e_Z^2 \end{bmatrix} =
\begin{bmatrix} d_{11}^{i+1} & d_{12}^{i+1} \\
 		d_{21}^{i+1} & d_{22}^{i+1} \end{bmatrix}
\begin{bmatrix} \partial_Y & 0 \\
                  0   & \partial_Z \end{bmatrix} +
\begin{bmatrix} \partial_Y & 0 \\
                  0   & \partial_Z \end{bmatrix}
\begin{bmatrix} d_{11}^i &d_{12}^i \\
                d_{21}^i & d_{22}^i \end{bmatrix} +
\begin{bmatrix} f_{11}^i &f_{12}^i \\
                f_{21}^i & f_{22}^i \end{bmatrix}
\]
Thus we see that $e_Y - e_Y^2 = d_{11}\partial_Y +
\partial_Yd_{11} + f_{11}$, and similarly for $e_Z$. Consequently,
in the category, $e_Y$ is the identity on $Y^*$, and $e_Z$ is both
idempotent and nilpotent on every term of  $Z^*$.
Now let $w = 1 + e_Z+ e_Z^2+ e_Z^3 + \dots$.
This is a chain map from $Z^*$ to itself. It is well defined
because on every term of $Z^*$ the sum is finite. On the other hand,
\[
e_Z  = (e_Z - e_Z^2) w = (d_{22}\partial_Z + \partial_Zd_{22} + f_{22})w
= d_{22}w\partial_Z +  \partial_Zd_{22}w + f_{22}w.
\]
Hence, $e_Z$ is the zero map in $\CK_{\ts+\Vsplt}(\bC)$. This proves that
the idempotent $e$ is split. The proof in the other categories is similar.
\end{proof}


\section{Acyclic complexes and localizations} \label{sec:acyclic}

In this section we introduce the derived categories that are the Verdier
localizations of the homotopy categories at thick subcategories of
acyclic complexes. Several variations on the standard theme are
discussed. It turns out that some different constructions yield the same
end object. The primary results concern the existence of the derived
categories and the idempotent completions.  Let $\bC$
be as in \ref{nota:cpx}.

We recall that a subcategory $\CL$ of a triangulated
category $\CC$ is {\it thick}
if it a full triangulated subcategory of $\CC$ and if it is closed under
taking direct summands. In this context, triangulated means close under the 
shift functor and if two of three objects in any triangle in $\CC$ are in 
$\CL$ then so is the third. The definition that we give has been called
Rickard's Criterion (see \cite{neeman}). It expresses precisely the conditions
needed to construct a Verdier localization of $\CC$ by inverting
any morphism such that third object in the triangle of that morphism
is also in subcategory $\CL$.

Given an exact category ($\CC, \CE)$, the subcategory
of acyclic objects $\FA(\CC, \CE)$ is the full subcategory in $\Cpx(\CC)$
consisting of all exact complexes of the form
\[
\xymatrix{
X^*: & \dots \ar[r] & X^{n-1} \ar[r] & X^n \ar[r] & X^{n+1} \ar[r] & \dots
}
\]
such that for every $n$, the map $X^n \to X^{n+1}$ decomposes
as a composition of an admissible epimorphism $X^n \to A^n$ with
an admissible monomorphism $A^n \to X^{n+1}$ where
$A^n \to X^{n+1} \to A^{n+1}$ is an exact sequence in $\CE$.

In a category $\bC$ of complexes over $kG$,
we are interested in some subcategories
of acyclic complexes. The first is the subcategory of all acyclic
complexes, where here acyclic means being
exact or simply having zero homology.
We denote it $\CA(\bC)$.
If $\bC = \Cpx(kG)$, it is an
easy check to see that this is the subcategory
$\FA(\Mod(kG), \seq(\Mod(kG))$
of all $kG$-modules where $\seq(\Mod(kG))$ denotes
all exact sequences.
If $\bC$ is $\cpx(kG)$, then the subcategory of acyclic complexes is
the collection
$\FA(\bfmod((kG), \seq(\bfmod(kG)))$ of the indicated exact category.

Conditions can be put on the types of acyclic complexes that are
acceptable. For example, let $V$ be a
fixed $kG$-module, which is free and finitely generated as a $k$-module,
and let
\[
\CA_{\Vsplt}(\Cpx(kG)) \ = \ \CA(\Mod((kG),
\Vsplt\thyph\seq(\Mod(kG)))
\]
where $\Vsplt\thyph\seq(\Mod(kG))$ is the collection
of $V$-split sequences of $kG$-modules.
In this case an acyclic object is a complex
of $kG$-modules that is both $V$-split and exact. Hence its
tensor product with $V$ is contractible.

Similarly, it is possible to define acyclic complexes
in the homotopy categories
$\CK(\bC)$, $\CK_{\Vsplt}(\bC)$ and $\CK_{\Vts}(\bC)$.
These are the classes of the corresponding
acyclic complexes in $\CA(\bC)$. That is, for example,
the objects in
$\CA(\CK_{\Vsplt}(\bC))$ are $V$-split-homotopy classes of
complexes in $\bC$ that are exact. .

\begin{prop} \label{prop:acyc1}

Assume that $V$ is free of finite rank as a module over $k$.
The subcategories  $\CA(\CK(\bC))$ and
$\CA_{\Vsplt}(\CK(\bC))$, are thick subcategories of
$\CK(\bC)$, for $\bC$ as in \ref{nota:cpx}.
\end{prop}

\begin{proof}
See Lemma 1.2 of \cite{neeman}.
\end{proof}

There is warrant for some care here.
It seems that the subcategory $\CA(\CK_{\Vsplt}(\bC))$ and
$\CA_{\Vsplt}(\CK_{\Vsplt}(\bC))$ are {\em not} thick subcategories
of $\CK_{\Vsplt}(\bC)$, in general. The problem here is that the notion
of acyclic is muddled. One can consider the image of the subcategory of
acyclic complexes of $\bC$ in the homotopy category. But the relative
projective $\Proj(\Vsplt$-$\seq)$ are not acyclic in the usual sense of
being exact. Hence, there exist
zero objects in the homotopy category that are not acyclic
in the usual sense.
Still we can prove the following.

\begin{prop} \label{prop:acyc2}
Assume that $V$ is free of finite rank as a $k$-module.
The subcategories $\CA(\CK_{\Vts}(\bC))$ and
$\CA_{\Vsplt}(\CK_{\Vts}(\bC))$ are thick subcategories
$\CK_{\Vts}(\bC)$.
\end{prop}
	
\begin{proof}
Because $\Proj(\Vts$-$\seq)$ consists of acyclic complexes (see Proposition
\ref{prop:VTsplitproj}), the subcategory $\CA(\CK_{\Vts}(\bC))$ is
triangulated. It is clear that it is closed under taking direct summands,
and hence, it is a thick subcategory.

The problem left is to show that the category $\CA_{\Vsplt}(\CK_{\Vts}(\bC))$
is triangulated.
Supppose that $X^*$ and $Y^*$ are $V$-split acyclic complexes and that
$f: X^* \to Y^*$ is a chain map. We need to show that the third object in
the triangle of $f$ is $V$-split and acyclic. The third object is
the object $B^*$, given by the diagram:
\[
\xymatrix{
\SE: \quad 0 \ar[r] & X^* \ar[r]^{\alpha} \ar[d]^{f} & X^* \otimes \CV^* \ar[r] \ar[d] &
\Omega^{-1}_{\SE}(X^*) \ar[r] \ar@{=}[d] & 0 \\
0 \ar[r] & Y^* \ar[r]  & B^* \ar[r] &
\Omega^{-1}_{\SE}(X^*) \ar[r] & 0
}
\]
where $\SE$ is in  $V$-$\ts$-$\seq$ and $B^*$ is the pushout in the left square.
Here $\CV^*$ is given in the proof of Proposition \ref{prop:VTsplitproj}.
To see that $B$ is a $V$-split acyclic complex, begin by tensoring the
entire diagram with $V$ and consider the upper row which is the exact sequence
\[
\xymatrix{
0 \ar[r] & X^* \otimes V \ar[r]^{\alpha \quad} &
X^* \otimes \CV^* \otimes V \ar[r] &
\Omega^{-1}_{\SE}(X^*) \otimes V \ar[r] & 0
}
\]
Now recall that the map $k \otimes V \to V \otimes V^\# \otimes V$ is
split. So the complex $\CV^* \otimes V$ is a direct sum of two
complexes one of which has the form
\[
\xymatrix{
\dots \quad 0 \ar[r] & V \ar[r] & V \ar[r] & 0 \quad \dots
}
\]
where the nonzero terms occur in degree -1 and 0. Moreover, this summand
contains the image of the chain map $\alpha$. Thus the first part of the
upper row of the previous diagram looks like
\[
\xymatrix{
& \dots \quad 0 \ar[r] & X^* \otimes V \ar[d]^\alpha \ar[r] & 0 \quad \dots \\
\dots \quad 0 \ar[r] & X^* \otimes V \ar[r]^{1 \otimes 1} &
X^* \otimes V \ar[r] & 0 \quad \dots
}
\]
Because $X^* \otimes V$ is a split sequence, a straightforward exercise shows
that the chain map $\alpha$ is also split. Consequently, the top row of the
original diagram when tensored with $V$ is a split sequence of complexes,
and the bottom row, being the pushout of the top row, when tensored with
$V$ is also split. Hence $B$ is $V$-split as asserted.

Thus we have that $\CA_{\Vsplt}(\CK_{\Vts}(\bC))$ is triangulated. It is
obviously closed under taking direct summands. Hence, it is a thick subcategory
of $\CK_{\Vts}(\bC).$
\end{proof}

If the category $\CA$ is the thick subcategory of
acyclic complexes in a homotopy
category $\CC$ of complexes, then the derived category $\CD_{\CA}(\CC)$
is the Verdier quotient or localization of $\CC$ at $\CA$. The objects
in the derived category are the same as those in $\CC$. But the morphisms
between two objects $X^*$ and $Y^*$ are obtained by inverting any
morphism such that the third object in the triangle of that morphism
is in $\CA$. Such a morphism is called a quasi-isomorphism.
Thus a morphism is a composition $g^{-1}f$ as in the
diagram
\[
\xymatrix{
X^* \ar[r]^f & Z^* & Y^* \ar[l]_g
}
\]
where the third object in the triangle containing $g$ is in $\CA$.

We use the notation $\CD_a(\CK_{\CX}(\bC))$ to mean the derived category of
$\CK_{\CX}(\bC)$ for $\CX$ one of $\ts$ or $V$-$\ts$,  with respect
to the subcategory of acyclic complexes $\CA_a$. Here $a$ is either
$-$ (blank)  or $\Vsplt$, meaning that
either all acyclic complexes or the $V$-split
acyclic complexes. Thus $\CK_{\Vts}(\Cpxb(kG))$ means
the quotient category $\Cpxb(kG))/(V$-$\ts$-$\Proj)$ of bounded complexes of
$kG$-modules by the projectives of the exact category
($\Cpxb(kG)$,$V$-$\ts$-$\seq$), and $\CD_{\Vsplt}(\CK_{\Vts}(\Cpxb(kG)))$
is its localization by inverting any map such that the third object
of the triangle of that map is an acyclic complex that splits on
tensoring with $V$. As we see below, the notation can be
simplified even more. Indeed, there is some contraction in the list
of derived categories.

\begin{prop} \label{prop:tsequiv}
Assume that $V$ is free of finite rank as a $k$-module.
Let $\bC$ be as in \ref{nota:cpx}. The natural
functor $\CK_{\Vts}(\bC) \to \CK(\bC)$ induces equivalences on the derived
categories $\CD(\CK_{\Vts}(\bC)) \to \CD(\CK(\bC))$ and
$\CD_{\Vsplt}(\CK_{\Vts}(\bC)) \to \CD_{\Vsplt}(\CK(\bC))$.
Moreover, these are triangular equivalences.
\end{prop}

\begin{proof}
Let $\ts_k\thyph\seq$ be the collection of sequence that are term split on
restriction to $k$. By Frobenius Reciprocity (see Theorem
\ref{thm:std-reptheory}(b)), $\ts_k\thyph\seq = kG\thyph\ts\thyph\seq$.
We have a sequence of subcategories
\[
\ts_k\thyph\seq\thyph\Proj \subseteq \Vts\thyph\seq\thyph\Proj \subseteq
\ts\thyph\seq\thyph\Proj,
\]
leading to a sequence of functors
\[
\xymatrix{
\CK_{\ts_k}(\bC) \ar[r] & \CK_{\Vts}(\bC) \ar[r] &
\CK_{\ts}(\bC).
}
\]
It is well known that the composition of the two functors induces an
equivalence between $\CK_{kG\thyph\ts}(\bC)$ and $\CK(\bC)$. The point is that
the projectives of each of the exact categories in question are acyclic
complexes by \ref{prop:VTsplitproj}. Any acyclic complex is quasi-isomorphic
to the zero complex and hence becomes the zero object in the derived
category. So for example, if a map between complexes factors through an
element of $\Proj(\Vts\thyph\seq)$, then it becomes the zero map in the
derived category. The same argument applies to prove that the first
statement of the theorem.

The equivalence
$\CD_{\Vsplt}(\CK_{\Vts}(\bC)) \to \CD_{\Vsplt}(\CK(\bC))$
is proved similarly. The fact that these functors induce triangle equivalences
is an exercise that we leave to the reader.
\end{proof}

\begin{theorem} \label{thm:derivedic}
All of the derived categories $\CD_a(\CK_b(\bC))$, that we have considered
and have countable direct sums or allow the
direct sum of a countable number of copies
of any object,  have idempotent completions.
\end{theorem}

\begin{proof}
This follows from \cite{bokstedt-neeman}, because the infinite categories have
countable coproducts (direct sums).
\end{proof}

\begin{theorem} \label{thm:KSbd}
Suppose that $k$ is a field.
The category $\CD^b(kG) = \CD(\CK(\cpxb(kG)))$ is a Krull-Schmidt category.
\end{theorem}

\begin{proof}
It is straightforward to show that $\CD^b(kG)$ is a hom-finite category,
and it has idempotent completions. Hence, if $X^*$ is a complex
in $\CD^b(kG)$, then its endomorphism ring is a finite dimensional
$k$-algebra that has a complete collection of primitive idempotents. Thus it
has a complete collection of primitive idempotents that sum to the identity.
This provides a decomposition of $X^*$ into indecomposable subcomplexes.
The uniqueness of the decomposition can be proved from the structure of
the endomorphism ring.
\end{proof}

Finally, we should note that all of the categories that have been discussed
respect the block structure of $kG$. The group algebra can be written as
a direct sum of indecomposable two-sided ideals $kG = B_1 \oplus \dots
\oplus B_n$. Each $B_i$ contains an idempotent $e_i$ which acts as the
identity of $B_i$, so that $B_i = e_i kG$
and $e_i = e_j$ for $j \neq i$. If $X$ is a $kG$-module or
complex of $kG$-modules then $X = \oplus_{i} e_iX$, and we say that $e_iX$ is
in the block $B_i$. There are no nonzero homomorphisms between modules or
complexes that are in different blocks.

For the record, we state the following.

\begin{theorem} \label{thm:blockcats}
  Suppose that $B$ is a block of $kG$. Let $\bC(B)$ be any of the
  categories of complexes in \ref{nota:cpx} or any of the categories of
 complexes of finite length modules as in Section \ref{sec:FL},
  restricted to modules in the block $B$.
The homotopy categories $\CK(\bC(B))$, $\CK_{\Vsplt}(\bC(B))$ and
$\CK_{\Vts}(\bC(B))$ as well as the derived categories
$\CD(\CK(\bC))$ and $\CD_{\Vsplt}(\CK(\bC(B)))$ are triangulated
categories. These categories have idempotent completions
provided they have countable direct sums or permit the countable
direct sum of an object with itself, or have objects that are complexes
of finite length modules.  If $k$ is a field, then
the categories $\cpxb(B)$, $\CK(\cpxb(B))$, $\CK_{\Vsplt}(\cpxb(B))$,
$\CK_{\Vts}(\cpxb(B))$, $\CD(\CK(\cpxb)) = \CD^b(B)$, and
$\CD_{\Vsplt}(\CK(\cpxb(B)))$ are Krull-Schimdt categories.
\end{theorem}

As a cautionary note, it should be added that seldom do any of the
above categories, that are associated to blocks, have a tensor structure.


\section{A functorial version of the Green correspondence} \label{sec:functor}
The purpose of this section is to lay a functorial framework for the
Green correspondence. The aim is to isolate, in an abstract way, the
condition necessary to define the correspondence. By this process, we
see that the correspondence can be defined in many contexts. Our
specific applications are to categories of complexes and their homotopy
categories and derived categories. The reader might notice that, even
though the setting is far more general, the
development here follows closely the same steps as in
the paper of Benson and Wheeler \cite{benson-wheeler}. Indeed that paper
was a big inspiration.

We wish to consider the following diagram of categories and functors.
In the diagram, all vertical arrows are inclusions of full subcategories.
For a category  $\CD$ the notation $\Ad(\CD)$ means the closure
of $\CD$ under taking direct summands. If $\CC$ and $\CD$ are subcategories
of $\CG$, then $\CC + \CD$ means the full subcategory of all objects
that can be written as the direct sum of an object in $\CC$ and an
object in $\CD$. 

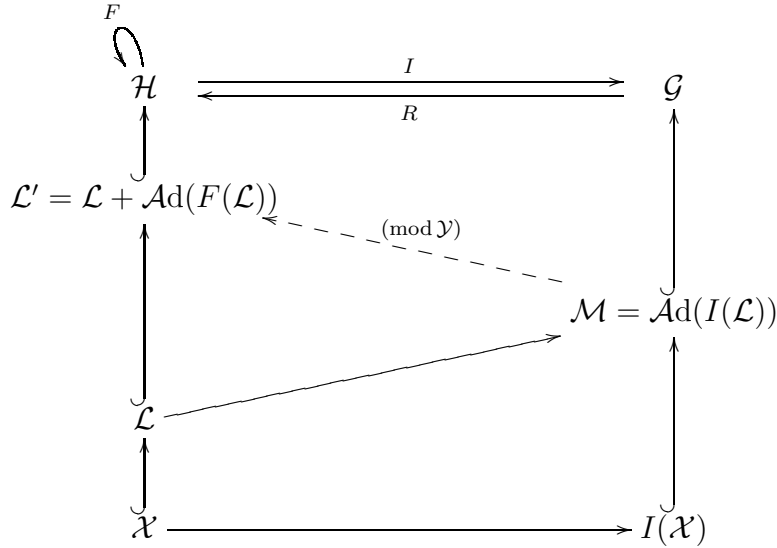
\[
\xymatrix@C+2pc{
\quad \CH \quad \ar@<.5ex>[rr]^{I} \ar@(u,ul)[]_{F}
&& \quad \CG \quad \ar@<.5ex>[ll]^{R} \\
\CL^\prime = \CL + \Ad(F(\CL)) \ar@{^{(}->}[u] \\
&& \CM = \Ad(I(\CL)) \ar@{-->}@<-.5ex>[ull]_{(\text{mod} \CY)}
\ar@{^{(}->}[uu]\\
\CL \ar@{^{(}->}[uu]  \ar[urr] \\
\CX \ar[rr] \ar@{^{(}->}[u] &&
   I(\CX) \ar@{^{(}->}[uu]
}
\]

Here $\CY = \CX + \Ad(F(\CL))$. The arrow from 
$\CM$ to $\CL^\prime$ is dashed because
it is not a functor, though there is a 
functor to $\CL^\prime/\CY$, as is explained below. 

Our main theorem is the following.

\begin{theorem}  \label{thm:functorial}
Suppose we have categories given as in the above diagram.
We assume the following.
\begin{enumerate}
\item All of the categories are additive categories.
\item The subcategories $\CL$ and $\CX$ are closed under direct sums and
summands in $\CH$.
\item The quotient categories $\CH/\CX$ and $\CG/I(\CX)$ are defined.
\item $I$ and $R$ are an adjoint pair of functors with natural transformations
$\varepsilon: 1_{\CH} \to RI$ and $\eta: IR \to 1_{\CG}$.
\item There is a functor $F: \CH \to \CH$ and a natural transformation
$\varphi: F \to RI$ such that the sum of the transformations
$(\varepsilon, \varphi): 1_{\CH} \oplus F  \to RI$ is an isomorphism
of functors. In particular, for every object $M$ in $\CH$
we have that $RI(M) \cong M \oplus F(M).$
\item For objects $L$ in $\CL$ and $M$ in $\CM$, every map $\gamma:
L \to F(M)$ and every map $\delta: F(M) \to L$
that factors through an object in
$\CY$, factors also through an object in $\CX$.
\item The quotient category $\CH/\CX$ has idempotent completions.
\end{enumerate}
Then we have an adjoint pair of $\bI$ and $\bR$ functors
\[
\xymatrix{
\CL/\CX \ar@<.5ex>[rr]^{\mathbb I} && \CM/I(\CX)
\ar@<.5ex>[ll]^{\mathbb R}
}
\]
that give categorical equivalences.
\end{theorem}

To define the functors, we require some preliminary information.
Throughout, we use the notation of the theorem.

\begin{lemma} \label{lem:gr1}
There exist functors $U$ and $V$,
\[
\xymatrix{
\CL/ \CX \ar@<.5ex>[rr]^{U} && \CL^\prime/\CY \ar@<.5ex>[ll]^{V}
}
\]
giving equivalences of categories.
\end{lemma}

\begin{proof}
First note that $\CL^\prime = \CL + \Ad F(\CM) = \CL + \CY$.
Define $U$ by $U(L) = L \oplus 0$ for $L \in \CL$, that is, the
functor induced by the inclusion of $\CL$ into $\CL^\prime$. It 
is clear that any map that factors through an object in $\CX$ also
factors through one in $\CY$.

For $V$, suppose that $L \oplus Y$ is
an object in $\CL +\CY$, {\it i. e.} $L$ in $\CL$ and 
$Y$ in $\CY$. Then, let $V(L\oplus Y) = L$. We first
check that this is well defined. For suppose that $L \oplus Y
\cong L^\prime \oplus Y^\prime$. The isomorphism between the two is
given by a matrix
\[
\begin{bmatrix} \alpha & \beta \\ \gamma & \delta \end{bmatrix}
\]
By condition (6) of the theorem, $\beta$ and $\gamma$ factor through
an element of $\CX$. It follows, with some computation,
that $L$ and $L^\prime$ are isomorphic modulo $\CX$. The other details
are likewise staightforward. Thus, $UV$ and
$VU$ are the identity functors.
\end{proof}

\begin{lemma}  \label{lem:gr2}
The category $\CL^\prime/\CY$ has idempotent completions.
\end{lemma}

\begin{proof}
Note that the category $\CL^\prime$ can not be assumed to have
idempotent completions. However we have seen that its quotient by $\CY$
is equivalent to $\CL/\CX$. Hence, it is sufficient to show that
$\CL/\CX$ has idempotent completions and we know this
by Condition (2) and (7) of the theorem. That is, $\CL/\CX \subseteq \CH/\CX$,
and the latter is idempotent complete.
\end{proof}

\begin{lemma} \label{lem:gr3}
Suppose that $L$ is in $\CL^\prime$, and that $M$ is a direct summand
of $L$.  Then there exists $Y$  in $\CY$ such that $M \oplus Y$ is in
$\CL^\prime$.
\end{lemma}

\begin{proof}
Suppose that $e: L \to L$ is the idempotent corresponding to $M$, the projection
of $L$ to $M$. By the previous lemma, we know that $e$ splits. So that
$L \cong L_e \oplus L_{1-e}$, where $L_e \cong M$ in $\CL^\prime/\CY$.
So there exist $Y$ and $Y^\prime$ im $\CY$ such that
$L_e \oplus Y^\prime  \cong M \oplus Y$ in $\CH$. Consequently,
$M \oplus Y$ is in $\CL^\prime$.
\end{proof}

\begin{corollary} \label{cor:functor}
The functor $R$ induces a functor $\hatR: \CM \to \CL^\prime/\CY$.
\end{corollary}

\begin{proof}
Suppose that $M$ is an object in $\CM$. Then $M$ is a direct summand of
$I(L)$ for some $L$ in $\CL$. Now, $RI(L) = L \oplus F(L)$, and
by the previous lemma, there exists $Y$ in $\CY$ such that $R(M) \oplus Y$
is in $\CL^\prime = \CL + \Ad (F(M))$. Thus, we define $\hatR(M)
= R(M) \oplus Y$. Note that this does not depend on the choice of $Y$.
\end{proof}

We are now ready to prove the main theorem of the section.

\begin{proof}[Proof of Theorem \ref{thm:functorial}]
We define  $\bI: \CL/\CX \to \CM/I(\CX)$
to be the functor induced by the restriction of $I$ to $\CL$.
The functor $\bR$ is the composition $\bR = V\hatR$.

Suppose that $M$ is an object in $\CM$. We know that there
exists $Y$ in $\CY$ such that $\hatR(M) = R(M) \oplus Y$ is in
$\CL^\prime$. That is $R(M) \oplus Y \cong L^\prime \oplus Y^\prime$
for $L^\prime$ in $\CL$ and $Y^\prime$ in $\CY$. Then $V\hatR(M) = L^\prime$.
The map $\hatR(M) = R(M) \oplus Y \to L^\prime \oplus Y^\prime$
has the form of a matrix
\[
\begin{bmatrix} \alpha & \beta \\ \gamma & \delta \end{bmatrix}
\]
where $\beta$ and $\gamma$ factor through objects in $\CX$. Consequently,
the relevant portion is the map $\alpha = \alpha_M: R(M) \to L^\prime$.

For $L$ in $\CL$ and $M$ in $\CM$, we define
\[
\xymatrix{
\gamma: \Hom_{\CL/I(\CX)}(\bI(L), M) \ar[r] &
\Hom_{\CM/\CX}(L, \bR(M))
}
\]
by $\gamma(f) = \alpha_M R(f) \varepsilon_L$. That is, this is the
composition modulo $\CX$
\[
\xymatrix@+.5pc{
L \ar[r]^{\varepsilon_L \quad} & RI(L) \ar[r]^{R(f)\quad} & R(M) \oplus Y
\ar[r] & L^\prime \oplus Y^\prime \ar[r] &  L^\prime
}
\]
Note here that if $f$ factors through $I(X)$ for $X$ in $\CX$,
then $\gamma(f)$ factors through an object in $\CX$. This happens
because any map from $L$ to $F(X)$ factors through an object in
$\CX$ by Condition (6) of the Theorem.

Now we define $\beta: \Hom_{\CM/\CX}(L, \bR(M)) \to
\Hom_{\CL/I(\CX)}(\bI(L), M)$ by letting $\beta(g) = \eta_MI(\alpha_M^{-1}g).$
That is, it is the composition
\[
\xymatrix@+2pc{
I(L) \ar[r]^{I(g) \qquad \qquad} & IV\hatR(M) \cong I(L^\prime \oplus Y^\prime)
\ar[r]^{\qquad \quad I(\alpha_M^{-1})} & IR(M) \ar[r]^{\eta_M} & M
}
\]
Of course, the map which we are calling $\alpha^{-1}$ is only an inverse
for $\alpha$ modulo maps that factor through elements of $\CX$. The isomorphism
$R(M) \oplus Y \cong L^\prime \oplus Y^\prime$ guarantees that there
is such a map. Note that if $g$ factors through an element of $\CX$, then
the composition $\eta_MI(\alpha^{-1}g)$ factors through an element
of $I(\CX)$.

Suppose that $g \in \Hom_{\CM/\CX}(L, \bR(M))$. Let $f = \beta(g)$.
Then,

\begin{align*}
\gamma(f) &=  \alpha_M R(f) \varepsilon_L =
\alpha_M R(\eta_MI(\alpha_M^{-1}g)) \varepsilon_L \\
&= \alpha_M R(\eta_M) RI(\alpha_M^{-1}g) \varepsilon_L
= \alpha_M \alpha_M^{-1} g = g
\end{align*}
modulo maps that factor through objects in $\CX$. The next to last step
in the above sequence of equations
is a consequence of the adjunction between $R$ and
$I$ which implies that $R(\eta_M) RI(\mu) \varepsilon_L = \mu$ for any
map $\mu: L \to R(M)$.

On the other hand if $f \in \Hom_{\CL/I(\CX)}(\bI(L), M)$, then
let $g = \gamma(f).$ So

\begin{align*}
\beta(g) & =  \eta_MI(\alpha_M^{-1}g) =
\eta_M I(\alpha_M^{-1}\alpha_M R(f) \varepsilon_L) \\
& = \eta_M IR(f) I(\varepsilon_L) = f
\end{align*}

Thus we have shown that $\gamma\beta$ and $\beta\gamma$ are the identities
and that $\bI$ and $\bR$ are an adjoint pair as asserted.
\end{proof}

\begin{remark}
The primary reason for the assumption of $\CH/\CX$  having idempotent
completions is to make possible the proof of Lemma \ref{lem:gr3}. That is,
we require that $\CL^\prime/\CY$ have idempotent completions as in
Lemma \ref{lem:gr2}. The same thing would be accomplished if we assumed
that $\CG$ and $\CH$ are Krull-Schmidt categories, or that $\CG$ and $\CH$
have countable direct sums.
\end{remark}


\section{Relative projective theory.}\label{sec:relprojtheory}
Let $k$ be a commutative ring, and suppose that $H$ is a subgroup
of a finite group $G$. In this section, we consider the induction and
restriction functors and remind the reader that many of the standard
results associated with these functors, by virtue of
their combinatorial nature,
hold for complexes and homotopy classes of complexes
as well as for modules. In particular, some of the conditions of
Theorem \ref{thm:functorial} are classical results in representation
theory. For notation,
let $\bC_G$ denote one of the categories of complexes such as $\Cpx$,
$\cpx$, $\Cpxp$ $\cpxFL$ of $kG$-modules or its homotopy categories
$\CK(\bC_G)$. There are induction and restriction functors, which we
denote $\Ind_H^G$ and $\Res_H^G$. When $X$ is a complex or module for
$kH$ and when there is little chance of confusion,
we often use the standard notation
\[
X^{\uparrow G} = \Ind_H^G(X) = kG \otimes _{kH} X
\]
to denote its induction to $G$.
We also use $X_{\downarrow H} = \Res_H^G(X)$ to
denote the restriction to $kH$  of a $kG$-complex
or module $X$.
For any object $X$ in $\bC$, there is a
natural decomposition
\[
\Ind_H^G(X) = X^{\uparrow G} = kG \otimes_{kH} X =
\oplus_{g \in G/H}^{} \quad g \otimes X
\]
as complexes of vector spaces, where the sum is over a
set of representatives of the left cosets of $H$ in $G$. Likewise
for a map $f: X \to Y$, $\Ind_H^G(f) = \sum_{g \in G/H} g \otimes f$.

The following results are standard in representation theory.
The statement (a) is usually called the Mackey Decomposition Theorem,
while (b) is Frobenius Reciprocity. The statement (c) which is a form
of Frobenius reciprocity, is often called the Eckmann-Shapiro Lemma. While
the proofs are classical, we give a quick review here in order to make it
clear that the theorems are valid for complexes regardless of the
coefficients. As the constructions are all well known and straightforward,
We leave it to the reader to check a great many details.

In the notation of the theorem below, we note that if $t \in G$, then
$(t \otimes \quad)$ is a functor from $\bC_{K^t \cap H}$ to
$\bC_{K \cap {}^tH}$, taking an object $X$ to $t\otimes X$. For $K$ 
a subgroup of $G$ and $t \in G$, ${}^tK = tKt^{-1}$ and $K^t = t^{-1}Kt$.

\begin{theorem} \label{thm:std-reptheory}
Suppose that $\bC$ is as in \ref{nota:cpx}  or a category of complexes of
finite length modules as in Section \ref{sec:FL}.
Let $H$ and $K$ be subgroups of $G$. Then the following hold in $\bC$.
\begin{enumerate}
\item[(a)]
There is a natural transformation of functors
\[
\alpha: \Res_K^G\Ind_H^G( \ \cdot \ ) \longrightarrow
\sum_{t \in K\backslash G/H}  \Ind_{K \cap {}^tH}^K
(t \otimes  \Res^H_{K^t \cap H}( \ \cdot \ ) )
\]
where the sum is over a set of representatives of the $K$-$H$-double
coset in $G$. The transformation is an isomorphism on objects.
Thus for an object $X$ in $\bC_H$, there is an isomorphism
\[
\alpha_X: (X^{\uparrow G})_{\downarrow K} \cong \oplus_{t\in K\backslash G/H}
(t \otimes X_{\downarrow K^t \cap H})^{\uparrow K}.
\]
\item[(b)]
Assume that the tensor products of objects in $\bC_G$ and $\bC_H$ are defined.
There is a natural transformation of functors
\[
\beta: \quad \cdot_G \ \otimes \Ind_H^G( \ \cdot_H \ ) \longrightarrow
\Ind_H^G(\Res_H^G ( \ \cdot_G \ )  \otimes  \ \cdot_H \ )
\]
from $\bC_G \times \bC_H$ to $\bC_G$, that is an isomorphism on objects.
Thus for objects $X$ in $\bC_G$ and $Y$ in $\bC_H$, $\beta_{X,Y}:
X \otimes Y^{\uparrow G} \cong (X_{\downarrow H} \otimes Y)^{\uparrow G}$.
\item[(c)] The functors $\Res_H^G$ and $\Ind^G_H$ are adjoints of each other.
\item[(d)] There is a functor $F: \bC_H \to \bC_H$ such that
$\Res_H^G\Ind_H^G = 1_{\bC_H} \oplus F$.
\end{enumerate}
\end{theorem}

\begin{proof}
The point of the Mackey Theorem (a) is that for any object $X$ in $\bC$
\[
\Ind_H^G(X) = X^{\uparrow G} = kG \otimes_{kH} X =
\oplus_{g \in G/H}^{} \quad g \otimes X
\]
where the sum is over a complete set of representatives
of the left cosets of $H$ in $G$. If one restricts to $K$, then the
sum over all of the left cosets in a single $K$-$H$-double coset is a
$kK$-subcomplex. It remains to show that
as $kK$-objects $\sum_{g \in KtH/H} g \otimes X
\cong \Ind_{K \cap tHt^{-1}}^K \Res^{tHt^{-1}}_{K \cap tHt^{-1}}(t \otimes X)$
where the sum is over a set of representatives of the left cosets of
$H$ that are contained in $KtH$ for $t \in G$. This proof is 
fairly straightforward.
It should be checked that the decomposition given by the Mackey Theorem
commutes with the differentials of a complex, and that the isomorphism,
which is defined internally on an object $X$, is actually a natural
transformation of the functors.

For the Frobenius reciprocity (b), there is a  map on objects that
sends
\[
X \otimes (kG \otimes_{kH} Y) \to
kG \otimes_{kH} (X_{\downarrow H} \otimes Y)
\]
defined by the formula
\[
x \otimes (g \otimes y) \mapsto g \otimes (g^{-1}x \otimes y),
\]
for $x \in X$, $y \in Y$ and $g \in G$.
The inverse isomorphism send $g \otimes (x \otimes y)$ to
$gx \otimes (g \otimes y)$.  It can be seen that the maps
commute with the differentials on the complexes and with maps
between complexes. Thus they are natural transformations of the
functors.

To prove (c), we define natural transformation $\eta: \Ind_H^G\Res_H^G
\to 1_{\bC_G}$ and $\varepsilon: 1_{\bC_H} \to \Res_H^G\Ind_H^G$, by
$\eta_X(g \otimes x) = gx$ for $g \in G$ and $x \in X$ in $\bC_G$, and
$\varepsilon_Y(y) = 1 \otimes y$ for $y \in Y$ in $\bC_H$.
Then the isomorphism
\begin{equation*}
\Hom_{\bC_H}(Y, \Res_H^G(X)) \to \Hom_{\bC_G}(\Ind_H^G(Y), X) \tag{adj1}
\end{equation*}
sends a map $f$ to $\eta \Ind_H^G(f)$, while the inverse isomorphism sends
$f$ to $\Res_H^G(f)\varepsilon$. Likewise we have natural transformation
$\eta^\prime: \Res_H^G\Ind_H^G \to 1_{\bC_H}$ and
$\varepsilon^\prime: 1_{\bC_G} \to \Ind_H^G\Res_H^G$ by
$\eta^\prime_Y(f)(\sum_{G/H} g \otimes y_g) = y_1$, for $y_g \in Y$
and the sum over a set of representatives of the left cosets of
$H$ in $G$, and $\varepsilon^\prime(x) = \sum_{G/H} g \otimes g^{-1}x$
for $x \in X$ in $\bC_G$. Then, the isomorphism
\begin{equation*}
\Hom_{\bC_H}(\Res_H^G(X), Y) \to \Hom_{\bC_G}(X, \Ind_H^G(Y)) \tag{adj2}
\end{equation*}	
takes $f$ to $\eta^\prime\Ind_H^G(f)$, while its inverse takes
a map $f$ to $R(f)\varepsilon^\prime$.

Statement (d) is a direct consequence of the Mackey Theorem, letting
\[
F( \ \cdot \ ) = \sum_{H\backslash G/H, x \neq 1} \Ind^H_{{}^tH\cap H}
(t \otimes \Res_{H\cap H^t}^{H}( \ \cdot \ ))
\]
where the sum is over a set or representatives of the $H\thyph H$-double
cosets in $G$ that are {\em not} the identity.
Note that $\varepsilon: 1_{\bC_H} \to \Res_H^G\Ind_H^G
= 1_{\bC_H} \oplus F$ is the injection and it is split by
$\eta^\prime: \Res_H^G\Ind_H^G \to 1_{\bC_H}.$
\end{proof}

For the homotopy categories we have the following.

\begin{theorem} \label{thm:def-homotop}
Let $H$ be a subgroup of $G$.
Let $\bC$ be a category of complexes as in Theorem \ref{thm:std-reptheory}.
Suppose that $V$ is a $kG$-module that is free of finite rank as a module
over $k$. Let $(\CK_*(\bC_G), \CK_*(\bC_H))$ be one of the pairs of 
homotopy categories ($\CK(\bC_G)$, $\CK(\bC_H))$,
($\CK_{\Vsplt}(\bC_G)$, $\CK_{V_{\downarrow H}\thyph\splt}(\bC_H)$),
($\CK_{V\thyph\ts}(\bC_G)$, $\CK_{V_{\downarrow H}\thyph\ts}(\bC_H)$), 
or ($\CK_{\ts+\Vsplt}(\bC_G)$, $\CK_{\ts+V_{\downarrow H}\thyph\splt}(\bC_H)$).
Then induction and restriction define
functors $\Ind^G_H: \CK_*(\bC_H) \to \CK_*(\bC_G)$ and
$\Res_H^G: \CK_*(\bC_G) \to \CK_*(\bC_H)$.
Moreover, these functors satisfy the conclusions of Theorem
\ref{thm:std-reptheory}
\end{theorem}

\begin{proof}
It suffices to show that the relative projective objects for the
homotopy are preserved by the functors. For the ordinary homotopy, a
relative projective object is a directs sum of two-term complexes of
the form $ \dots \to 0 \to W \to W \to 0 \to \dots$. It is obvious that
the induction or restriction of such a complex has the same form.
Consequently, the induction or restriction of a map that factors through
such a complex also factors through a relative projective.

Notice that the restriction of a $V$-projective complex of $kG$-modules
to $H$ is a $V_{\downarrow H}$-projective complex.
On the other hand, if $X$ is
a $V_{\downarrow H}$-projective module or complex, then $X$ is a direct
summand of $Y \otimes V^\#_{\downarrow H}  \otimes V_{\downarrow H}$ for
some object $Y$. By Frobenius Reciprocity
(Theorem \ref{thm:std-reptheory}(b)), $X^{\uparrow G}$ is a direct summand of
$Y^{\uparrow G} \otimes V^\# \otimes V$. Hence the induction of a
relative $V_{\downarrow H}$-projective object is a relative
$V$-projective object. This proves that the induced functors are defined.
Furthermore, we have the following commutative diagram:
\[
\xymatrix{
\bC_H \ar@<.5ex>[r]^{\Ind_H^G} \ar[d]_{q_H}
& \bC_G \ar[d]^{q_G} \ar@<.5ex>[l]^{\Res^G_H}  \\
   \CK_*(\bC_H)\ar@<.5ex>[r]^{\Ind_H^G}
& \CK_*(\bC_G)\ar@<.5ex>[l]^{\Res^G_H}
.}
\]
Since the functors $q_H$ and $q_G$ preserve direct sums,
the statements (a) follows.

If $X$ and $Y$ are isomorphic in $\bC_G$, then $X$ and $Y$ are
isomorphic in $\CK(\bC_G)$. Then the statement (b) follows from
the definition of the induction functors above.
With the above commutative diagram, the statement (c) and (d) follows from the
above discussion and from the explicit description of the isomorphisms
given in the proof of statement (c) and (d) of Theorem
\ref{thm:std-reptheory}.
\end{proof}

We remind the reader that if $H$ is a subgroup of $G$, then an object
in $\bC_G$ is (relatively) $H$-projective if it is a direct summand of
$Y^{\uparrow G}$ for some object $Y$ in $\bC_H$. By Frobenius reciprocity
(see part (b) above) this is the same as being
$V$-projective for $V = (k_H)^{\uparrow G}$. If $\CD$ is a collection of
subgroups of $G$, we say that an object is $\CD$-projective if it is
a summand of a direct sum of $D$-projective objects
for $D \in \CD$. This is the same as
$V$-projective for $V =\sum_{D \in \CD} k_D^{\uparrow G}$. A map is
$\CD$-projective if it factors through a $\CD$-projective object.
For a subgroup $H$ of $G$ and complexes $X$ and $Y$,
there is a relative
trace map $\Tr_H^G: \Hom_{\bC_H}(X, Y) \to \Hom_{\bC_G}(X, Y)$ given by
$\Tr_H^G(f) = \sum_{G/H} gfg^{-1}.$
Observer that, if $\alpha \in \Hom_{kG}(W, X)$ and
$\beta \in \Hom_{kG}(Y, Z)$ for objects $W$ and $Z$, then
$\beta\Tr_H^G(f)\alpha = \Tr_H^G(\beta f \alpha)$.

It is clear that the map $\Tr_H^G$ is well defined in any of the
categories of complexes, {\it i. e.} whenever a map of objects is a
map or sets. Some more care is needed for the homotopy categories.

\begin{lemma}
Let $\boK_G$ be one of the homotopy categories $\CK_*(\bC_G)$ as in
Theorem \ref{thm:def-homotop}. Then the map $\Tr_H^G$ induces a map
$\Tr_H^G: \Hom_{\boK_H}(X_{\downarrow H}, Y_{\downarrow H})
\to \Hom_{\boK_G}(X,Y)$ for any objects $X$ and $Y$.
\end{lemma}

\begin{proof}
Let $\CP_H$ be the collection of projective objects in $\bC_H$ relative to
the homotopy.  It suffices to show is that if $f: X \to Y$ is a
$kH$-map of objects in $\bC_G$
that is zero on $\boK_H$ (that is, factors through an object in
$\CP_H$) then $\Tr_H^G(f)$ factors through an object in $\CP_G$.
If $f$ factors through an object in $\CP_H$, then
it factors through a relative projective cover for $Y_{\downarrow H}$
which is in $\CP_H$. 
It can be seen from the description of the projectives
as in Propositions \ref{prop:Vsplitproj}, \ref{prop:VTsplitproj},
and \ref{prop:V+Tsplitproj}, that the restriction to $H$ of a 
relative projective cover 
for an object $Y$ can serve as a relative projective
cover for $Y_{\downarrow H}$.
Let $\beta: P \to Y$ be such a cover for $Y$, with $P$ in $\CP_G$.
Then we have that $f = \beta\alpha$ for $\alpha: X_{\downarrow H}
\to P_{\downarrow H}$. But then $\Tr_H^G(f) = \Tr_H^G(\alpha)\beta$
factors through an object in $\CP$.
\end{proof}

\begin{theorem}  \label{thm:std-maps}
Suppose that $\bC$ is as in \ref{thm:std-reptheory} or a
homotopy category of a category of such complexes
as in Theorem \ref{thm:def-homotop}. Let $H$ be a subgroup of
$G$ and let $\CD$ be a collection of subgroups of $G$.
Then the following hold in $\bC$.
\begin{enumerate}
\item[(a)] If an object $X$ in $\bC$ is $H$-projective, then
$\Id_X = \Tr_H^G(\mu)$ for some $kH$-map $\mu: X \to X$.
\item[(b)] A map $f: X \to Y$ factors through an $H$-projective
object if and only $f = \Tr^G_H(\mu)$ for some $\mu: X_{\downarrow H}
\to Y_{\downarrow H}$ in $\bC_H$.
\item[(c)] Assuming the $\bC_G$ has idempotent completions,
an object $X$ is $H$-projective if and only if $X$ is a
direct summand of $\Ind_H^G\Res_H^G(X)$.
\item[(d)] A map $f: X \to Y$ is $\CD$-projective if and only if
$f = \sum_{D \in \CD} \Tr_D^G(\gamma_D)$ where for $D \in \CD$, 
$\gamma_D: X_{\downarrow D} \to Y_{\downarrow D}$ is a $kD$-map.
\item[(e)] If $f:X \to Y$ is $\CD$-projective and $g:Y \to Z$ is
$\CE$-projective, for $\CE$  another collection of subgroups of $G$,
then $gf$ is $\CF$-projective where $\CF = \{D \cap E \vert D \in \CD,
E \in \CE\}$.
\item[(f)] Suppose that  $k$ is a field and $B$ is a block of $kG$
having defect group $Q$ and if $X$ is an object in $\bC$ and in the
block $B$, then $X$ is $Q$-projective.
\end{enumerate}
\end{theorem}

\begin{proof}
Because we have the commutative diagram:
\[
\xymatrix{
  \Hom_{\bC_H}(X_{\downarrow H}, Y_{\downarrow H})\ar[r]^{\Tr_H^G} \ar[d]  & \Hom_{\bC_G}(X,Y) \ar[d]   \\
   \Hom_{\boK_H}(X_{\downarrow H}, Y_{\downarrow H})\ar[r]^{\Tr_H^G}  & \Hom_{\boK_G}(X,Y)
,}
\]
it suffices to prove the Theorem in category of complexes.

If $X = \Ind_H^G(Y)$ for $Y \in \bC_H$, then, writing
$X = \sum_{G/H}(g \otimes Y)$, define $f: X \to X$ by $f(1 \otimes y) =
1\otimes y$ and $f(g\otimes y) = 0$ if $g \not\in H$. Then
$\Id_X = \Tr_H^G(f)$. If $X$ is only a direct summand of $Y^{\uparrow G}$,
then $\Id_X$ factors through $\Id_{Y^{\uparrow G}}$. Hence it is a
relative trace in this case also.  This proves statement (a).

To prove (b) we first note that if $f:X \to Y$ factors through an
$H$-projective object, then it factors through $Z^{\uparrow G}$ for
some object $Z$ in $\bC_H$, and it is a composition with the identity
of $Z^{\uparrow G}$ which is a relative trace. For the converse, suppose
that $f = \Tr_H^G(\mu)$ for some $\mu: X_{\downarrow H} \to Y_{\downarrow H}$.
then define $\sigma: X \to (Y_{\downarrow H})^{\uparrow G}$, by
$\sigma(x) = \sum_{g \in G/H} g \otimes \mu(g^{-1}x)$ for $x \in X$, and
$\tau: (Y_{\downarrow H})^{\uparrow G} \to Y$ by $\tau(g \otimes y) = gy$
for $y \in Y$. Then note that $\tau\sigma = \Tr_H^G(\mu)$ in $\bC_G$.

For (c) we notice in the above proof that the condition that $\Id_X =
\Tr_H^G(\mu)$ for some $\mu$ implies the existence of
$\sigma: X \to (X_{\downarrow H})^{\uparrow G}$ and
$\tau: (X_{\downarrow H})^{\uparrow G} \to X$ with $\tau\sigma = \Id_X$.
Thus $\sigma\tau$ is an idempotent in the endomorphism ring of
$(X_{\downarrow H})^{\uparrow G}$. Assuming that it splits, $X$ is a
direct summand of $(X_{\downarrow H})^{\uparrow G}$.

Statement (d) follows from (b). For (e), notice that if $H$ and $J$ are
subgroups then a composition $\Tr_H^G(\alpha)\Tr_J^G(\beta) =
\Tr_H^G(\alpha \Res_H^G(\Tr_J^G(\beta)))$. The remainder of the proof is
is a consequence of the Mackey Theorem and the transitivity of induction.

Similarly, the last statement is a consequence of the fact that the
central idempotent that is the identity for the block $B$ is a relative
trace from the defect group $D$. See a standard text such as \cite{feit}
\end{proof}


\section{The Green correspondence} \label{sec:applications}
The purpose of this section is present a version of the Green correspondence
for derived categories and categories of complexes associated to group
algebras. The classical Green correspondence assumes that
$k$ is a field of characteristic $p$ and that $H$ is a subgroup
containing the normalizer of a $p$-subgroup and the correspondences is
between relatively $\FX$-projective $kG$-modules and relatively
$\FY$-projective $kH$-modules for certain collections of subgoups
$\FX$ and $\FY$. The approach here uses Theorem \ref{thm:functorial}
and is somewhat more general as far as the choices of the collections
of subgroups.

Suppose that $H$ is a subgroup of the finite group $G$. Let $\FP$ be
a nonempty collection of subgroups of $H$. We define two collections
of subgroups of $H$:
\[
\FX = \{ gP_1g^{-1} \cap P_2 \ \vert \ P_1, P_2 \in \FP \ \text{and} \
g \in G \setminus H\}
\]
and
\[
\FY = \{ gPg^{-1} \cap H \ \vert \ P \in \FP \ \text{and} \
g \in G \setminus H\}.
\]
Let
\[
V_{\FX} \ = \ \sum_{P \in \FX} k_P^{\uparrow G}
\quad \text{and} \quad
V_{\FY} \ = \ \sum_{P \in \FY} k_P^{\uparrow G}
\]
Note here that $V_{\FX}$ and $V_{\FY}$ are both free and finitely
generated as modules over the coefficient ring $k$.

Let $\bC_G$ denote any one of the categories
$\Cpx(kG)$, $\cpxFL(kG)$, $\Cpx^*(kG)$, or $\cpxFL^*(kG)$
for $* = +, -$ or $b$. Likewise, we let $\bC_H$ be the
same with $G$ replaced by $H$.
Let $\boK_G = \CK_*(\bC_G)$ be one of the homotopy categories $\CK(\bC_G)$,
$\CK_{\Vsplt}(\bC_G)$,
$\CK_{V\thyph\ts}(\bC_G)$ or $\CK_{\ts+\splt}(\bC_G)$
(where $V$ is a finitely generated $kG$-module that is free
as a $k$-module). Then let
$\boK_H$ be $\CK(\bC_H)$,
$\CK_{V_{\downarrow H}\thyph\splt}(\bC_H)$,
$\CK_{V_{\downarrow H}\thyph\ts}(\bC_G)$ or
$\CK_{\ts+V_{\downarrow H}\splt}(\bC_G)$, to correspond to $\boK_G$.
Let $\CC_G$ be $\bC_G$ or $\boK_G$ for some choice.

In $\CC_G$, let  $\FX$-$\Proj(\CC_G)$ be the collection of
$V_{\FX}$-projective objects. Such an object
is a direct summand of a direct sum
of objects induced from objects in $\CC_P$ for $P \in \FX$.
Similarly, an exact sequence of objects in $\CC_G$ is
$\FX$-split, if it is $V_{\FX}$-split, thus implying that
the sequence splits on restriction to {\em every} subgroup
$P \in \FX$.

The idea expressed in the following lemma is used to establish
idempotent completions in the proofs of some theorems.

\begin{lemma} \label{lem:simplequocats}
Let $\bC$ be a category of complexes as above. Let V be a finitely
generated $kG$-module that is free as a module over $k$. For a
collection of subgroups $\FU$, let
$V_{\FU} = \sum_{Q \in \FU} (k_Q)^{\uparrow G}$. Let
$\CK = \CK_{\ts+\Vsplt}(\bC).$
Then
\[
\frac{\CK}{\FU\thyph\Proj(\CK)} \quad = \quad
\CK_{\ts+(V\oplus V_{\FU})\thyph\splt}(\bC).
\]
\end{lemma}

\begin{proof}
Suppose that $X$ and $Y$ are objects in $\bC$ and $\theta: X \to Y$ is a
morphism in $\CK$. Then $\theta$ if $\FU$-projective if and only if
$\theta = \beta\alpha$, $\alpha: X \to Z$, $\beta: Z \to Y$, where
$Z$ is $\FU$-projective and $\alpha$, $\beta$ are morphisms in $\CK$.
But then $\theta - \beta\alpha$ is zero in $\CK$ and hence factors
through an object that is  projective relative 
to $\ts+\Vsplt$ sequences. This means
that $\theta$ is a projective relative to $\ts+(V\oplus V_{\FU})\thyph\splt$
sequences.
\end{proof}

The main theorem is the following.

\begin{theorem}  \label{thm:greencomplex}
Let $\CC_G$, $\CC_H$, $\FP$, $\FX$, $\FY$ be as above. Then there are
equivalences of categories
\[
\begin{Large}
\xymatrix{
\frac{\FP\thyph\Proj(\CC_H)}{\FX\thyph\Proj(\CC_H)} \ar@<.5ex>[rr]^{\CI} &&
\frac{\FP\thyph\Proj(\CC_G)}{\FX\thyph\Proj(\CC_G)}  \ar@<.5ex>[ll]^{\CR}
}
\end{Large}
\]
that are induced by the induction and restriction operations.
\end{theorem}

\begin{proof}
The proof is by application of Theorem \ref{thm:functorial}. The problem
is to show that the hypothese hold in every case.
The setup is that $\CH = \CC_H$, $\CG = \CC_G$, $R = \Res_H^G$,
$I = \Ind_H^G$, $\CL = \FP\thyph\Proj(\CC_H)$, By transitivity of induction,
$\CM = \FP\thyph\Proj(\CC_G).$ Notice that $\FP\thyph\Proj(\CC_H)$ is
closed under taking direct summands by definition. Finally,
$\FX = \FX\thyph\Proj(\CC_H)$, while $\Ad(I(\FX)) = \FX\thyph\Proj(\CC_G)$

Conditions (1), (2) and (3) of Theorem \ref{thm:functorial} are clearly
satisfied.  Conditions (4) and (5), follow from Theorem
\ref{thm:std-reptheory} parts (c) and (d), respectively, and Theorem
\ref{thm:def-homotop}.

For condition (6), we need a lemma, which says
that a subgroup of some element of $\FP$ that is also a subgroup of some
element in $\FY$ is contained in a subgroup in $\FX$.
This is a standard result. That is, if $Q \subseteq P_1$ for $P_1 \in \FP$
and $Q \subseteq H \cap gP_2g^{-1}$ for $P_2 \in \FP$ and $g \not\in H$,
then $Q \subseteq P_1 \cap gP_2g^{-1} \in \FX$. If $L \in \CL$,
$M \in \CM$ and $\gamma: L \to F(M)$ factors through an $\FY$-projective
object, then $\gamma = \gamma\Id_L$ factors through an $\FX$-projective
object by statements (d) and (e) of Theorem \ref{thm:std-maps}.

To prove (7), it is only necessary to show that
any of the quotient categories
$\CU = \CC_H/\FX\thyph\Proj(\CC_H)$ has idempotent completions.
Note that $\CU$ is a triangulated category. In every case that
we consider, by Lemma \ref{lem:simplequocats}, $\CU$ is a
category that has been discussed in Section \ref{sec:FL} with regard
to the question of idempotent completions. Thus,
$\CU$ has idempotent completions by Proposition \ref{prop:idemcomp} 
and Theorem \ref{thm:idemfinleng}.
\end{proof}

As is pointed out in \cite{benson-wheeler}, the functors are not
precisely the restriction and induction functors. The problem is that
the restriction of a $\FP$-projective $kG$-module is not $\FP$-projective
as a $kH$-module. So the inverse of the induction functor is actually the
composition of the restriction with another categorical equivalence
(called ``V'' in Lemma \ref{lem:gr1}) 
as in the proof of Theorem \ref{thm:functorial}.

\begin{theorem}\label{thm:greenacyclic}
Let $\CC_G$, $\CC_H$, $\FP$, $\FX$, $\FY$ be as above.
Let either $\bA_G$ be $\CA(\CC_G)$ and $\bA_H$ be $\CA(\CC_H)$  
or $\bA_G = \CA_{\Vsplt}(\CC_G)$ and $\bA_H = \CA_{V_H\thyph\splt}(\CC_H)$.
Assume that $\bA_G$ is a thick subcategory of
$\CC_G$ and $\bA_H$ is a thick subcategory of
$\CC_H$ (see the remark following Proposition \ref{prop:acyc1}).
Then there are equivalences of categories
\[
\begin{Large}
\xymatrix{
\frac{\FP\thyph\Proj(\bA_H)}{\FX\thyph\Proj(\bA_H)} \ar@<.5ex>[rr]^{\CI} &&
\frac{\FP\thyph\Proj(\bA_G)}{\FX\thyph\Proj(\bA_G)}  \ar@<.5ex>[ll]^{\CR}
}
\end{Large}
\]
that are induced by the induction and restriction operations.
\end{theorem}

\begin{proof}
The categories involved are additive, thick subcategories of $\CC_G$ or
$\CC_H$. The restriction and induction of an acyclic complex is again an
acyclic complex. Regarding the hypothesis of Theorem \ref{thm:functorial},
conditions (1), (2) and (4) are automatic. For condition (3), it should
be noted that if $f: X \to Y$ is a map of objects in $\bA_H$ that factors
through an $\FX$-projective object, then it factors through
the relative $\FX$-projective cover
of the object $Y$, which is an acyclic object. Hence, there are no
$\FX$-projective maps in $\CC_H$ of objects in $\bA_H$, that are not also
$\FX$-projective in $\bA_H$. The same hold of $I(\FX)$-projective maps
between objects in $\bA_G$. These facts require some checking on a case
by case basis depending on the category $\CC$. We leave the check to the
reader.

Condition (5) of Theorem \ref{thm:functorial} is essentially the Mackey
theorem which holds for acyclic complexes. It should be noted that the
functor $F$ takes acyclic objects to acyclic objects.
For Condition (6), we observe
that if $F: L \to F(M)$ factors through an
$\FY$-projective object where $L$ is an acyclic complex in
$\FP\thyph\Proj(\bA_H)$ and $M$ is in
$\FP\thyph\Proj(\bA_G)$, then it factors through the $\FY$-projective
cover of $F(M)$ which is acyclic. Because $X$ is $\FP$-projective,
the map $f$ factors through an $\FX$-projective complex, which we can
take to be acyclic as before.

Finally, there is the question of Condition (7). However, as the subcategory
$\bA_H$ is thick in $\CC_H$, the property of the quotient category of
$\CC_H$ by the $\FX$-projective objects having idempotent completions,
extends to the category of acyclic objects. That is, the splitting of an
idempotent on an acyclic object in the quotient of $\CC_H$ gives the direct
sum of two objects that must be acyclic.
\end{proof}

Another approach to a proof for the above theorem is that category $\bA_G$ 
and $\bA_H$ are subcategories of $\bC_G$ and $\bC_H$, and the induction 
and restriction functors for $\bA$ are the restrictions of those for 
$\bC$. Moreover, an object in $\bA_H$ is $\FX$-projective in $\bA_H$ 
if and only if it is $\FX$-projective in $\bC_H$. So the question might be
whether the induced equivalences in Theorem \ref{thm:greencomplex} 
extends to those of Theorem \ref{thm:greenacyclic}. The latter theorem 
asserts an affirmative answer and the real reason is embedded in the 
proof. Essentially, it is that a map between acyclic objects in 
$\bA_H$, that factors through an $\FX$-projective objective in 
$\bC_H$, factors through an $\FX$-projective object in $\bA_H$. It is a 
consequence of the fact that it factors through a projective cover. 

For derived categories we come down to the following.

\begin{theorem}  \label{thm:greenderived}
Let $\CD_G = \CD(\CC_G)$ with $\CD_H = \CD(\CC_H)$ or
$\CD_G = \CD_{\Vsplt}(\CC_G)$ with $\CD_H =
\CD_{V_{\downarrow H}\thyph\splt}(\CC_H)$ for $\CC_G$, $\CC_H$ as in the
previous theorem. Assume, as in that theorem, that the subcategories of
acyclic objects are thick. 
Then there are equivalences of categories
\[
\begin{Large}
\xymatrix{
\frac{\FD\thyph\Proj(\CD_H)}{\FX\thyph\Proj(\CD_H)} \ar@<.5ex>[rr]^{\CI}  &&
\frac{\FD\thyph\Proj(\CD_G)}{\FX\thyph\Proj(\CD_G)} \ar@<.5ex>[ll]^{\CR}
}
\end{Large}
\]
that are induced by the induction and restriction operations.
\end{theorem}

\begin{proof}
The derived categories $\CD_G$ and $\CD_H$ have the same objects as 
$\CC_G$ and $\CC_H$, respectively, and we know that we have 
well defined equivalences on objects. 
It is easy to see that the induction functor takes exact sequences of
complexes to exact sequences of complexes and in the homotopy categories,
and takes triangles to triangles. This from $\bC_H$ to $\bC_G$. By the 
equivalences, the same happens for the inverse.  
In the derived category $\CD_G$ or $\CD_H$, 
a morphism between objects $X$ and $Y$ 
is an equivalence class of diagrams
\[
\xymatrix{
X & Z \ar[l]_\phi \ar[r]^\theta &Y
}
\]
where the third object in the triangle (in $\CC_G$ or $\CC_H$ as 
appropriate) of $\phi$ is acyclic.
Because, the functors take triangles to triangles and acyclic objects
to acyclic objects, they are equivalences also on morphisms. 
So we have equivalences of the derived categories as additive 
categories. 
\end{proof}

\begin{remark}
The proof of the above theorem avoids the question of idempotent 
completeness of any of the derived categories. We know that idempotent 
completions do exist in several cases. Balmer and Schlichting
\cite{balmer-schlichting} verify  idempotent completions in the bounded
derived category of an exact category that has idempotent completions, and 
also for the category of bounded below complexes. However, in general,
the localization of an idempotent complete triangulated category, by an 
idempotent complete thick subcategory, may not be idempotent complete. 
\end{remark}


\section{Blocks and triangulations} \label{sec:tri-equiv}
Suppose that $k$ is an algebraically closed
field of characteristic $p >0$, or a complete discrete valuation 
ring whose residue field is an algebraically closed field of 
characteristic $p>0$. 
Let $\bC$ be one of the categories of
complexes as before.
We remind the reader that all of the categories that we have discussed
respect the block structure. That is, if $B$ and $B^\prime$ are two
different blocks of $kG$ then there are 
no nonzero morphism from any complex of
modules in $B$ to any complex of modules in $B^\prime.$  
This is simply because the idempotents for the blocks, which act as 
identity on modules in the block, annihilate each other.  A block $B$
has a defect group $Q$ with the property that every module or complex in
$B$ is $Q$-projective, {\it i. e.} is a direct summand of an object
induced from $Q$. The same holds for complexes of $B$-modules, and
in fact, every morphism between two modules or complexes in $B$ is in the 
image of the relative trace map $\Tr_Q^G$.

The Brauer correspondent of $B$ is a block of $kN_G(Q)$,
with the property that the product of the central idempotent of $B$ and $b$
is not zero. 
But more importantly we have the following. We set the notation, as
this will be used again.

\begin{notation} \label{nota:blocks}
For a block $B$ of $kG$, let $Q$ be its defect group and let $b$ be the
Brauer correspondent of $B$. Let $\FP = \{ Q \}$, $H = N_G(Q)$,
$\FX = \{Q \cap Q^\sigma \vert \sigma \in G\setminus H\}$,
and $\FY = \{H \cap Q^\sigma \vert \sigma \in G\setminus H\}$. For $\bC$ a 
category of complexes such as $\Cpx$ or $\Cpxp$, let $\bC(B)$ be the 
full subcategory of $\bC_G$ consisting of those complexes that lie in $B$.
\end{notation}

Note that in the above notation, $\FD\thyph\Proj(B) = \bC(B)$.

\begin{prop} \label{prop:blockimages}
Use Notation \ref{nota:blocks}.
Suppose that $B$ is a block of $kG$.
Let $\bC$ be a category of complexes as in Theorem \ref{thm:greencomplex}. 
For the functors
\[
\begin{large}
\xymatrix@-.3pc{
f: \frac{\FD\thyph\Proj(\bC_G)}{\FX\thyph\Proj(\bC_G)} \ar[r] &
\frac{\FD\thyph\Proj(\bC_H)}{\FY\thyph\Proj(\bC_H)}
\quad \text{and} \quad 
g: \frac{\FD\thyph\Proj(\bC_H)}{\FY\thyph\Proj(\bC_H)} \ar[r] &
\frac{\FD\thyph\Proj(\bC_G)}{\FX\thyph\Proj(\bC_G)}
}
\end{large}
\]
we have that
\[
\begin{large}
\xymatrix@-.3pc{
f:\frac{\bC(B)}{\FX\thyph\Proj(\bC(b))} \ar@{^{(}->}[r] &
\frac{\bC(b)}{\FY\thyph\Proj(\bC(b))}  \quad \text{and} \quad
g: \frac{\bC(b)}{\FY\thyph\Proj(\bC(b))} \ar@{^{(}->}[r] &
\frac{\bC(B)}{\FX\thyph\Proj(\bC(B))}
}
\end{large}
\]
\end{prop}

\begin{proof} It follows from Theorem 2.7 in \cite{harris}. The proof in 
\cite{harris} is only for finitely generated modules, 
but it holds equally well for 
complexes and homotopy classes of complexes.  
\end{proof}

It is worth noting the following. Its proof follows from Theorem 
\ref{thm:std-maps}(e), 
after recalling that any object in $\bC(b)$ is $\FD$-projective. 

\begin{lemma} \label{lem:blockrel}
With the above notation, we have that 
\[ 
\frac{\bC(b)}{\FY\thyph\Proj(\bC(b))} = \frac{\bC(b)}{\FX\thyph\Proj(\bC(b))}.
\]
That is, a map between objects in $\bC(b)$ factors through an $\FY$-projective
object if and only it factors through an $\FX$-projective object. 
\end{lemma}

The main theorem of the section is the following.

\begin{theorem} \label{thm:homotop-triang}
Use Notation \ref{nota:blocks}. Suppose that $B$ is a block of $kG$. 
Suppose that $\bC$ is a category of complexes or a homotopy category
as in Theorem \ref{thm:greencomplex}. Then the
equivalences
\[
\begin{Large}
\xymatrix{
\frac{\bC(b)}{\FX\thyph\Proj(\bC(b))} \ar@<.5ex>[rr]^{\CI} &&
\frac{\bC(B)}{\FX\thyph\Proj(\bC(B))} \ar@<.5ex>[ll]^{\CR}
}
\end{Large}
\]
are equivalences of triangulated categories.
\end{theorem}

\begin{proof}
Let $V = \sum_{P \in \FX} k_P^{\uparrow G}$. The object of the
proof is to show that given a triangle in the domain category, its
image is a triangle in the target category. To this suppose that
$\sigma:X^* \to Y^*$ is a map of complexes in $\bC(B)$. That is, it is a
complex in $\bC_G$ whose terms are all in the block $B$. We view the
quotient of the homotopy category as in Theorem \ref{lem:simplequocats}.
We construct the triangle of $\sigma$ by constructing the diagram:
\begin{equation} \label{eq:triangle}
\xymatrix{
0 \ar[r] & X^* \ar[r] \ar[d]^\sigma & I^* \ar[r] \ar[d]&
\Omega_S^{-1}(X^*) \ar[r] \ar@{=}[d] & 0 \\
0 \ar[r] & Y^* \ar[r]               & E^* \ar[r]       &
\Omega_S^{-1}(X^*) \ar[r]            & 0
}
\end{equation}
Here, the upper row is the first step in relative injective resolution of
$X^*$, and the complex $E^*$ is the pushout of the upper left corner.
In this particular case, it means that $I^*$ is the direct sum of a
split complex and a complex of $V$-projective modules, and the upper row
is term split and split on restriction to every subgroup in $\FX$.

Now consider the effect of the restriction functor $f$ on the diagram.
The restriction of an $\FX$-projective complex
is $\FY$-projective. The restriction of a split complex remains split,
the restriction of the upper row remains term-split. Finally, the
upper row splits on restriction to any subgroup of $\FY$. Thus, the
restriction to $H$ of the triangle
\[
\xymatrix{
X^* \ar[r] & Y^* \ar[r] & B^* \ar[r] & \Omega_S^{-1}(X^*)
}
\]
is again a triangle.
\end{proof}

To extend the theorem to the derived category, we need the next lemma.

\begin{theorem} \label{thm:deriv-tri-equiv} Let $\CD_G$, $\CD_H$ be 
as in Theorem \ref{thm:greenderived}
Suppose that $\CD(B)$ is the full subcategory of $\CD_G$ consisting 
of those classes of complexes in the block $B$. Let $\CD(b)$ be the 
same for the Brauer correspondent of $B$. Then we have equivalences 
of triangulated categories
\[
\begin{Large}
\xymatrix{
\frac{\CD(B)}{\FX\thyph\Proj(\CD(B)} \ar@<.5ex>[rr]^{\CI} &&
\frac{\CD(b)}{\FX\thyph\Proj(\CD(b)})  \ar@<.5ex>[ll]^{\CR}
}
\end{Large}
\]
\end{theorem}

\begin{proof}
The triangles in the derived category are the same as those in the 
homotopy category. Consequently the theorem is proved by application
of Theorems \ref{thm:greenderived} and \ref{thm:homotop-triang}. 
\end{proof}

Suppose that $S$ is a Sylow $p$-subgroup of $G$.
Let $\FX = \{S \cap S^x \vert x \in G \setminus N_G(S)\}$.
This is the collection of nontrivial Sylow intersections. A theorem of
J. A. Green says that the defect group of any block is a Sylow
intersection. Consequently, the defect group of any block is either
equal to $S$ or is in $\FX$. This information allows the following
observation.

\begin{theorem} \label{thm:tensor-equiv}
Let $\bC$ be one of the categories of complexes or homotopy classes
of complexes that as in 
Theorem \ref{thm:greencomplex} or a derived category as in 
Theorem \ref{thm:greenderived}. Assume that $\bC_G$ is a tensor category.
Let $H = N_G(S)$.
Then the equivalences
\[
\begin{Large}
\xymatrix{
\frac{\bC_H}{\FX\thyph\Proj(\bC_H)} \ar@<.5ex>[rr]^{\CI}  &&
\frac{\bC_G}{\FX\thyph\Proj(\bC_G)} \ar@<.5ex>[ll]^{\CR}
}
\end{Large}
\]
are equivalences of tensor triangulated triangles.
\end{theorem}

\begin{proof}  Let $B_1, \dots, B_t$ be the blocks of $kG$ that have
$S$ as defect group. Then notice that
\[
\frac{\bC_G}{\FP\thyph\Proj(\bC_G)} = \sum_{i = 1}^t
\frac{\bC(B_i)}{\FP\thyph\Proj(\bC(B_i))}
\]
since any module or complex in any block with smaller 
defect group is $\FP$-projective.
Thus by Theorem \ref{thm:homotop-triang}, these are triangle equivalences.
Hence, the only question here is the tensor structure.
These categories have tensor structures because the subcategories
being factored out are closed under arbitrary tensor product.
That is, for example, if $X^*$ is in
$\SE-Proj$ for $\SE = \ts$+$V$-$\seq$, then so is $X^* \otimes Y^*$
for any appropriate $Y^*$. So finally, the proof is the observation
that the restriction map commutes with tensor products.
\end{proof}

\begin{remark} 
As we previously noticed, there is no tensor product
of arbitrary objects in $\cpx(kG)$. 
Howeve, one should still be able to use the tensor structure
for the category $ \cpx(kG)$ in some constructive way.
\end{remark}

\begin{remark}
It might be tempting to use the above result to accomplish something such
as classifying thick subcategories or localizing subcategories. However,
such structures may be very complicated and it is likely that the Balmer
spectrum of thick subcategories is not Noetherian. See \cite{carlson}.
\end{remark}

\bibliographystyle{amsplain}

\begin{thebibliography}{10}
\bibitem{atiyah} M. Atiyah, {\em On the Krull-Schmidt Theorem with
appplications to sheaves}, Bull. Soc. Math. France, {\bf 84}(1956), 307\--317.

\bibitem{auslander-carlson} M. Auslander and J. F. Carlson, {\it Almost split
sequences and group algebras}, J. Algebra, {\bf 103}(1986), 122\--140.

\bibitem{auslander-kleiner} M. Auslander and M. Kleiner, {\em Adjoint
functors and an extension of the Green correspondence for group
representations}, Adv. Math. \textbf{104}(1994), 297--314.

\bibitem{balmer-schlichting} P. Balmer and M. Schlichting, {\it
Idempotent completion of triangulated categories},
J. Algebra {\bf 236}(2001), 819--834.

\bibitem{benson-carlson}
D. J. Benson and J. F. Carlson, {\it Nilpotent elements in the
Green ring}, J. Algebra, {\bf 104}(1986), 329\--350.

\bibitem{benson-wheeler} D. J. Benson and W. W. Wheeler, {\em The Green
correspondence for infinitely generated modules},
J. London Math. Soc. (2)  63  (2001),  no. 1, 69--82.

\bibitem{bokstedt-neeman} M. B\"okstedt and A. Neeman, {\it Homotopy limits in
triangulated categories}, Compositio Math. {\bf 86}(1993), 209-234.

\bibitem{carlson} J. F. Carlson,
{\em Thick subcategories of the relative stable category},
in "Geometric and Topological Aspects of the Representation Theory of
Finite Groups”, Springer Proceedings in Mathematics and Statistics,
Springer, 2018..

\bibitem{carlson-peng-wheeler} J. F. Carlson, C. Peng
and W. Wheeler, Transfer Maps and Virtual Projectivity,
{\em J. Algebra}, 204(1998), 286-311.

\bibitem{feit} W. Feit, The representation theory of
finite groups, North-Holland Publishing Co.,
Amsterdam-New York, 1982.

\bibitem{green1}J. A. Green,  {\em A transfer theorem for
modular representations}, J. Algebra
\textbf{1}, (1964), 73-84.

\bibitem{green2}J. A. Green,  {\em Relative module categories
for finite groups}, J. Pure Appl. Algebra
\textbf{2}  (1972), 371-393.

\bibitem{grime}M. Grime, {\em Triangulated Categories,
Adjoint Functors, and
Bousfield Localization}, A dissertation submitted
to the University of Bristol  for the degree of
Doctor of Philosophy, 2005.

\bibitem{happel} D. Happel, Triangulated categories
in the representation theory of finite-dimensional
algebras, {\em London Mathematical Society Lecture
Note Series}, 119. Cambridge University Press, Cambridge, 1988.

\bibitem{harris} M. Harris, {\em A block extension of categorical
results in the Green correspondence context}, J. Group Theory
{\bf 17}(2014, 1117--1131.

\bibitem{keller} B. Keller, {\em Chain complexes and
stable categories}, Manuscripta Math.  {\bf 67}
(1990),  no. 4, 379-417.

\bibitem{keller2} B. Keller, {\em Derived categories and their uses}, 
in ‘‘Handbook of Algebra,’’ Vol. 1, pp. 671-- 701, M.  
Hazewinkel, Ed., Elsevier, Amsterdam, 1996.

\bibitem{neeman} A. Neeman, Triangulated Categories,
Annals of Mathematics Studies, 148. Princeton
University Press, Princeton, NJ, 2001.

\bibitem{okuyama} T. Okuyama, {\em A generalization of projective
covers of modules over finite group algebras}, unpublished manuscript.

\bibitem{quillen} D. Quillen, Higher algebraic K-theory. I.
Algebraic K-theory, I: Higher K-theories (Proc. Conf.,
Battelle Memorial Inst., Seattle, Wash., 1972),  pp. 85-147.
{\em Lecture Notes in Math.}, Vol. 341, Springer, Berlin 1973.

\bibitem{wang-zhang} L. Wang and J. Zhang, {\em Stable Green equivalences in
bounded derived categories}, J. Pure Appl. Algebra, {\bf 220}(2016),
3498--3513.
\end{thebibliography}

\end{document}